\documentclass[journal]{IEEEtran}
\usepackage{cite}
\usepackage[pdftex]{graphicx}
\usepackage{epstopdf}

\usepackage{amsmath}
\usepackage{amsfonts}
\usepackage{amsthm}
\usepackage{bbm}

\theoremstyle{plain}
\newtheorem{thm}{Theorem}

\newtheorem{prop}{Proposition}
\newtheorem{cor}{Corollary}
\newtheorem{remark}{Remark}

\newtheorem{cla}{Claim}

\usepackage{algorithm}
\usepackage{algorithmic}
\usepackage{multirow}  
\usepackage{booktabs}  

\usepackage{array}

\usepackage{graphicx} 
\ifCLASSOPTIONcompsoc
\usepackage[caption=false,font=normalsize,labelfont=sf,textfont=sf]{subfig}
\else
\usepackage[caption=false,font=footnotesize]{subfig}
\fi
\usepackage{color}

\usepackage{stfloats}
\usepackage{bm}

\usepackage{amsthm,amsmath,amssymb}
\usepackage{mathrsfs}

\allowdisplaybreaks[4]

\renewcommand{\arraystretch}{1.3}

\usepackage{nomencl}
\usepackage{ifthen}
\renewcommand{\nomgroup}[1]{%
\ifthenelse{\equal{#1}{A}}{\item[\emph{\textbf{Abbreviations}}]}{%
\ifthenelse{\equal{#1}{B}}{\item[\emph{\textbf{Sets and Indices}}]}{%
\ifthenelse{\equal{#1}{C}}{\item[\emph{\textbf{Parameters in Pre-Dispatch Stage}}]}{%
\ifthenelse{\equal{#1}{D}}{\item[\emph{\textbf{Decision Variables in Pre-Dispatch Stage}}]}{%
\ifthenelse{\equal{#1}{E}}{\item[\emph{\textbf{Parameters and Decision Variables in Re-Dispatch Stage}}]}{%
	\ifthenelse{\equal{#1}{F}}{\item[\emph{\textbf{XX}}]}
}
				}
			}
		}
	}
}

\makenomenclature

\usepackage{verbatim} 

\usepackage{makecell}
\usepackage{threeparttable} 

\usepackage[figuresright]{rotating}

\begin{document}
%

\title{Data-Driven Two-Stage Distributionally Robust Dispatch of Multi-Energy Microgrid}

%
%

\author{Xunhang~Sun,~\IEEEmembership{Graduate~Student~Member,~IEEE},
        Xiaoyu~Cao,~\IEEEmembership{Member,~IEEE},
        Bo~Zeng,~\IEEEmembership{Member,~IEEE},\\
        Miaomiao Li, 
        Xiaohong~Guan,~\IEEEmembership{Life~Fellow,~IEEE},
        and~Tamer~Ba\c{s}ar,~\IEEEmembership{Life~Fellow,~IEEE}
        \vspace{-1cm} 
\thanks{This work was partially supported by the National Key R\&D Program of China under Grant 2022YFA1004600, and by NSFC under Grant 624B2111, Grant 62373294, and Grant 62192752. \emph{(Corresponding author: Xiaoyu Cao.)}}%
\thanks{X. Sun, X. Cao, and X. Guan are with the School of Automation Science and Engineering, Xi’an Jiaotong University, Xi’an 710049, Shaanxi, China (e-mail: xhsun@sei.xjtu.edu.cn; cxykeven2019@xjtu.edu.cn; xhguan@xjtu.edu.cn).}%
\thanks{B. Zeng is with the Department of Industrial Engineering, University of Pittsburgh, Pittsburgh, PA 15106 USA (e-mail: bzeng@pitt.edu).}%
\thanks{M. Li is with the National Innovation Platform (Center) for Industry-Education Integration of Energy Storage Technology, Xi’an Jiaotong University, Xi’an 710049, Shaanxi, China (e-mail: mmiaoli@stu.xjtu.edu.cn).}%
\thanks{T. Ba\c{s}ar is with the Coordinated Science Laboratory, University of Illinois Urbana-Champaign, Urbana, IL 61801 USA (e-mail: basar1@illinois.edu).}%
}

\maketitle


\begin{abstract}
This paper studies adaptive distributionally robust dispatch (DRD) of the multi-energy microgrid under supply and demand uncertainties. A Wasserstein ambiguity set is constructed to support data-driven decision-making. By fully leveraging the special structure of worst-case expectation from the primal perspective, a novel and high-efficient decomposition algorithm under the framework of  column-and-constraint generation is customized and developed to address the computational burden. Numerical studies demonstrate the effectiveness of our DRD approach, and shed light on the interrelationship of it with the traditional dispatch approaches  through  stochastic programming and robust optimization schemes. Also, comparisons with popular algorithms in the literature for  two-stage distributionally robust optimization verify the powerful capacity of our algorithm in computing the DRD problem.
\end{abstract}

\begin{IEEEkeywords}
Distributionally robust optimization, Wasserstein metric, column-and-constraint generation, microgrid dispatch. \vspace{-15pt}
\end{IEEEkeywords}

%
\IEEEpeerreviewmaketitle

\printnomenclature
\vspace{-10pt}

\nomenclature[C]{$T/t/\Delta_t$}{Total number/index/duration of time slots}

\nomenclature[C]{$c_{\rm e,buy}^t/c_{\rm e,sell}^t$}{Electricity pruchasing/selling price}
\nomenclature[C]{$c_{\rm g,buy}$}{Hydrogen procurement price}

\nomenclature[C]{$c_{\rm elz}^{\rm om}/c_{\rm fc}^{\rm om}$}{Unit operation and maintenance cost of ELZ/FC}
\nomenclature[C]{$c_{\rm bss}^{\rm deg}$}{Unit degradation cost of BSS}

\nomenclature[C]{$c_{\rm elz,c}^{\rm su}/c_{\rm elz,c}^{\rm sd}$}{Unit cold startup/shutdown cost of ELZ}
\nomenclature[C]{$c_{\rm elz,w}^{\rm su}/c_{\rm elz,w}^{\rm sd}$}{Unit warm startup/shutdown cost of ELZ}
\nomenclature[C]{$c_{\rm fc}^{\rm su}/c_{\rm fc}^{\rm sd}$}{Unit startup/shutdown cost of FC}

\nomenclature[C]{$\overline{P}_{\rm wt}/\overline{P}_{\rm pv}/\overline{P}_{\rm bss}/\overline{P}_{\rm sub}$}{Capacity of wind turbine (WT)/photovoltaic (PV) system/BSS/substation}%

\nomenclature[C]{$\overline{P}_{\rm elz}/\underline{P}_{\rm elz}$}{Maximum/minimum power input of ELZ in production state}%
\nomenclature[C]{$\overline{P}_{\rm fc}/\underline{P}_{\rm fc}$}{Maximum/minimum power output of FC}%
\nomenclature[C]{$P_{\rm elz,s}$}{Power consumption of ELZ in standby state}%

\nomenclature[C]{$\overline{E}_{\rm bss}/\underline{E}_{\rm bss}$}{Maximum/minimum state-of-charge of BSS}%

\nomenclature[C]{$\eta_{\rm bss,c}/\eta_{\rm bss,d}$}{Charging/discharging efficiency of battery storage system (BSS)}%

\nomenclature[C]{$\eta_{\rm elz}/\eta_{\rm fc}$}{Efficiency of ELZ/FC}%

\nomenclature[C]{$LHV_{\rm H_2}$}{Hydrogen low heat value}%

\nomenclature[C]{$\eta_{\rm elz,r}/\eta_{\rm fc,r}$}{Heat recovery efficiency of electrolyzer (ELZ)/fuel cell (FC)}%

\nomenclature[C]{$P_{\rm d}/M_{\rm d}$}{Nominal electricity/heat demand}%

\nomenclature[C]{$\phi_{\rm wt}^{t}/\phi_{\rm pv}^{t}$}{Time-varing factor of WT/PV output}%
\nomenclature[C]{$\phi_{\rm ed}^{t}/\phi_{\rm hd}^{t}$}{Time-varing factor of electricity/heat demand}%

\nomenclature[C]{$\nu_{\rm ht}/\nu_{\rm hwt}$}{Dissipation factor of hydrogen/hot water tank}%
\nomenclature[C]{$\overline{M}_{\rm ht}/\overline{N}_{\rm hwt}$}{Capacity of hydrogen/hot water tank}%

\nomenclature[C]{$\tau_{\rm cold}$}{Time delay of cold startup of ELZ}%

\nomenclature[C]{$\overline{H}_{{\rm buy}}/\overline{U}_{{\rm g,buy}}$}{Limit on hydrogen purchase quantity/number}%

\nomenclature[D]{$p_{\rm buy}^t/p_{\rm sell}^t$}{Electricity pruchasing and selling power}
\nomenclature[D]{$h_{\rm buy}^t$}{Purchased hydrogen energy from market}

\nomenclature[D]{$p_{\rm bss,c}^t/p_{\rm bss,d}^t$}{Charging/discharging power of BSS}

\nomenclature[D]{$u^{t}_{\rm bss}$}{State of BSS: 1-charging; 0-discharging}%
\nomenclature[D]{$e^{t}_{\rm bss}$}{State-of-charge of BSS}%

\nomenclature[D]{$u^{t}_{\rm e,buy}$}{State of electricity transaction: 1-buying; 0-selling}%
\nomenclature[D]{$u^{t}_{\rm g,buy}$}{State of hydrogen purchase: 1-yes; 0-no}%
\nomenclature[D]{$u^{t}_{\rm fc}$}{State of FC: 1-on; 0-off}%
\nomenclature[D]{$u^{t}_{\rm elz,p}/u^{t}_{\rm elz,s}/u^{t}_{\rm elz,i}$}{States of ELZ: $u^{t}_{\rm elz,p}=1$-production; $u^{t}_{\rm elz,s}=1$-standby; $u^{t}_{\rm elz,i}=0$-idle}%
\nomenclature[D]{$y^{t}_{\rm fc}/z^{t}_{\rm fc}$}{Actions of FC: $y^{t}_{\rm fc}=1$-startup; $z^{t}_{\rm fc}=1$-shutdown}%
\nomenclature[D]{$y^{t}_{\rm cold}/z^{t}_{\rm cold}/y^{t}_{\rm warm}/z^{t}_{\rm warm}$}{Actions of ELZ: $y^{t}_{\rm cold}/y^{t}_{\rm warm}=1$-cold/warm startup; $z^{t}_{\rm cold}/z^{t}_{\rm warm}=1$-cold/warm shutdown}%

\nomenclature[D]{$p_{\rm wt}^{t}/p_{\rm pv}^{t}$}{Power output of WT/PV system}%

\nomenclature[D]{$p_{\rm elz,p}^t/p_{\rm elz,s}^t$}{Component of power input of ELZ in production/standby state}
\nomenclature[D]{$p_{\rm elz}^t/g_{\rm elz}^t/m_{\rm elz}^t$}{Power input/hydrogen outflow rate/heat outflow rate of ELZ}
\nomenclature[D]{$p_{\rm fc}^t/g_{\rm fc}^t/m_{\rm fc}^t$}{Power output/hydrogen inflow rate/heat outflow rate of FC}

\nomenclature[D]{$h_{{\rm ht}}^{t}/n_{{\rm ht}}^{t}$}{Hydrogen/heat energy storage level of hydrogen/hot water tank}%

\nomenclature[D]{$m_{{\rm hwt}}^{t}$}{Heat flow of hot water tank}%

\nomenclature[E]{$\tilde{\cdot}$}{Re-dispatch counterpart of parameter and variable of pre-dispatch stage}%
\nomenclature[E]{$\tilde{p}_{\rm loss}^t/\tilde{m}_{\rm loss}^t$}{Unmet electricity/heat deman}%
\nomenclature[E]{$\iota_{\rm e}/\iota_{\rm h}$}{Unit compensation for unmet electricity/heat demand}%
\nomenclature[E]{$\Delta \cdot$}{Deviation variable}%

\section{Introduction}
%
%
%
%

\IEEEPARstart{C}{onsistent} with the target of global carbon neutrality, the multi-energy microgrid (MEMG) \cite{10722871} offers a promising paradigm for low-carbon, efficient, and reliable energy provision by integrating local renewable energy (RE) generators, storage systems, conversion devices, and multiple energy loads. However, RE outputs are highly dependent on meteorological conditions, while energy consumption is inherently variable, leading to significant operational risks. Hence, a dispatch approach that can effectively address various uncertainties is essential for MEMG's supply-demand balance.

The most classical and widely adopted approaches to handling uncertainty in decision-making are stochastic programming (SP) \cite{stochastic1} and robust optimization (RO) \cite{ben2009robust}. To date, they have been extensively applied to MEMG's dispatch (e.g., \cite{10722871,10026631,8967039,zhao2025robust}). SP requires complete knowledge of the probability distribution of random factors, and optimizes the objective in an expectation manner. In \cite{10722871}, a stochastic scheduling approach has been proposed for rural MEMG, considering uncertainties in RE generation and agricultural factors. Under the framework of Stackelberg game, \cite{10026631} has studied stochastic transactive energy management for the interaction between MEMG operator and multi-type users. Nevertheless, in practice, it is usually difficult to obtain the exact distribution, which is the primary limitation of SP. As for RO, it disregards all distributional information about the uncertainties, except for an uncertainty set, and seeks the optimal solution under the worst-case scenario. For hybrid AC/DC MEMG of a ship, a robust coordination approach has been studied in \cite{8967039} to confirm the safety of voyage under uncertain onboard loads, outdoor temperature, and solar power. A robust dynamic MEMG dispatch approach has been proposed in \cite{zhao2025robust} based on hierarchical model predictive control. However, we note that RO inherently leads to over-conservative solutions.

To cope with the aforementioned shortcomings of SP and RO, \textit{distributionally robust optimization} (DRO) \cite{rahimian2022frameworks} has been developed in recent years. Rather than assuming a fixed probability distribution as SP, DRO takes into account unknown perturbations in the underlying distribution, and postulates that the distribution lies in an \textit{ambiguity set}. Further, borrowing the ideas in RO, DRO hedges against the perturbation by optimizing the expected objective under the worst-case distribution. Hence, DRO can be viewed as a unification, as well as a trade-off, of SP and RO. Due to the consideration of partial distributional information, the solution derived by DRO is more robust compared to that by SP, yet remaining less conservative than that by RO. In fact, if the ambiguity set contains solely the true distribution, DRO reduces to SP; if it is large enough that includes all possible distributions, DRO reduces to RO. We notice that DRO has been preliminarily adopted in the literature on MEMG's dispatch, e.g., \cite{9770492,ma2024distributionally}. Nonetheless, the interrelationship between the DRO-based dispatch approach with the traditional SP- and RO-based ones has not been fully investigated.

The ambiguity set is a key component of DRO. Based on the modeling methods for ambiguous distributions, it can be broadly categorized into two groups, i.e., moment-based  and discrepancy-based. The moment-based ambiguity set, which is the primary focus of the early work on DRO, contains all distributions whose moments satisfy some given constraints \cite{delage2010distributionally}. However, it follows an assumption that certain information of the moments are known, which is not often the case in MEMG's practice. Also, the moment-based ambiguity set may lead to conservative decisions \cite{gao2023distributionally}, misaligning with the target of managing MEMG by using DRO. The discrepancy-based one contains all distributions that are close to a reference distribution with respect to a pre-specified discrepancy measure \cite{mohajerin2018data}. By leveraging empirical data, it enables \textit{data-driven} decision-making and has become the emphasis of current DRO research. Besides, it provides an opportunity to control the \textit{level of robustness} \cite{rahimian2022frameworks,mohajerin2018data}. Because of these advantages, the discrepancy-based ambiguity set appears to be well-suited for MEMG's dispatch. Of note, although there are multiple ways to measure the discrepancy, the \textit{Wasserstein metric} has gained increasing popularity in the literature on data-driven DRO due to its favorable theoretical properties \cite{mohajerin2018data,zhao2018data,yue2022linear}.

In the DRO family, two-stage DRO exhibits a strong and flexible modeling capacity by introducing an ingenious recourse decision that unfolds across stages and adapts to  uncertainty. However, solving two-stage DRO problems has long been a significant challenge. Although some decomposition algorithms have been proposed, e.g., in \cite{gamboa2021decomposition,saif2021data,duque2022distributionally}, they may not perform well in complex and practical-scale cases due to the lack of exploiting DRO's unique structure, i.e., the worst-case expectation. Recently, based on classical column-and-constraint generation (C\&CG) \cite{zeng2013solving}, an innovative algorithm framework has been proposed in \cite{lu2024}, showing a powerful computational capacity for general two-stage DRO problems. We believe it should contribute to the specific field of MEMG's dispatch, which generally requires for high computational efficiency in engineering practice.

Therefore, this paper studies the optimal dispatch problem of MEMG considering uncertainties on both the supply and demand sides. Specifically, we focus on a practical situation of MEMG where the distribution of random factors is not known exactly, while a limited set of empirical data is available. To hedge against potential operational risks, a data-driven  distributionally robust dispatch (DRD) formulation is proposed based on two-stage DRO. The first-stage problem determines the day-ahead pre-dispatch schedule, and the second-stage, adapting to the realizations of random factors, optimizes the intra-day re-dispatch decisions. A continuous \textit{Wasserstein ambiguity set} is built, containing all probability distributions around the \textit{empirical distribution}, to support data-driven decision-making. By analyzing the DRD model from the \textit{primal} perspective, we propose a novel reformulation of the worst-case expectation problem. It inspires us to design a column generation (CG)-based solution algorithm that guarantees finite-step convergence and facilitates parallel computing. Then, integrating CG as an inner subroutine in the classical C\&CG framework, an exact and high-efficient decomposition algorithm is developed and then customized according to the specifications of MEMG, referred to as \textit{C\&CG-DRO(CG)}, which demonstrates a superior capacity in computing the complete DRD problem. In comparison to the literature, the contributions of this paper can be summarized as follows:
\begin{enumerate}
	\item A two-stage adaptive DRD formulation is proposed for MEMG, incorporating both supply and demand uncertainties. A convex Wasserstein ambiguity set containing all probability distributions around an empirical distribution is employed to achieve data-driven decision-making. 
	\item To overcome the computational challenges, by fully leveraging the special sturcture of worst-case expectation from the primal perspective, a C\&CG-DRO(CG) algorithm with a feature of parallel computation is customized and developed to derive an exact and high-efficient solution to our DRD problem.
	\item Numerical studies demonstrate the advantages of our DRD approach over traditional SP- and RO-based dispatch approaches, elucidate their interrelationships, and verify the dominance of the computational power of C\&CG-DRO(CG) over those of existing methodologies.
\end{enumerate}

The remainder of this paper is organized as follows. Section \ref{formulation} proposes the DRD formulation for a typical MEMG based on two-stage DRO. A CG-based algorithm is customized in Section \ref{wcep-cg} for the worst-case expectation problem, integrating which C\&CG-DRO(CG) is developed in Section \ref{ccg}. Then, numerical studies are carried out in Section \ref{case} to demonstrate the effectiveness of our model and algorithms. Finally, in Section \ref{conclusion}, some conclusions are drawn. 

Unless explicitly noted otherwise, in the paper, DRO refers specifically to two-stage DRO. All vectors are column vectors.  $[N]$ represents the integer set of $\{1,\cdots,N\}$.

\section{Problem Formulation}\label{formulation}

We consider a typical MEMG as illustrated in Fig. \ref{fig1}. The MEMG is composed of RE units, including wind turbines (WTs) and photovoltaic (PV) systems, battery storage systems (BSSs), electrolyzers (ELZs), hydrogen tanks (HTs), fuel cells (FCs), and hot water tanks (HWTs). Our DRD approach possesses a two-stage framework. In the first-stage, the day-ahead pre-dispatch schedule is optimized under a forecasted basic scenario. Then, the second-stage re-dispatch problem is introduced to access how the pre-dispatch scheme can respond to varying operating conditions for intra-day operations. Moreover, the operating horizon (i.e., one day) is uniformly divided into $T$ slots, each lasting $\Delta_t$ units. 

\begin{figure}[]
	\centering
	\includegraphics[width=0.4\textwidth]{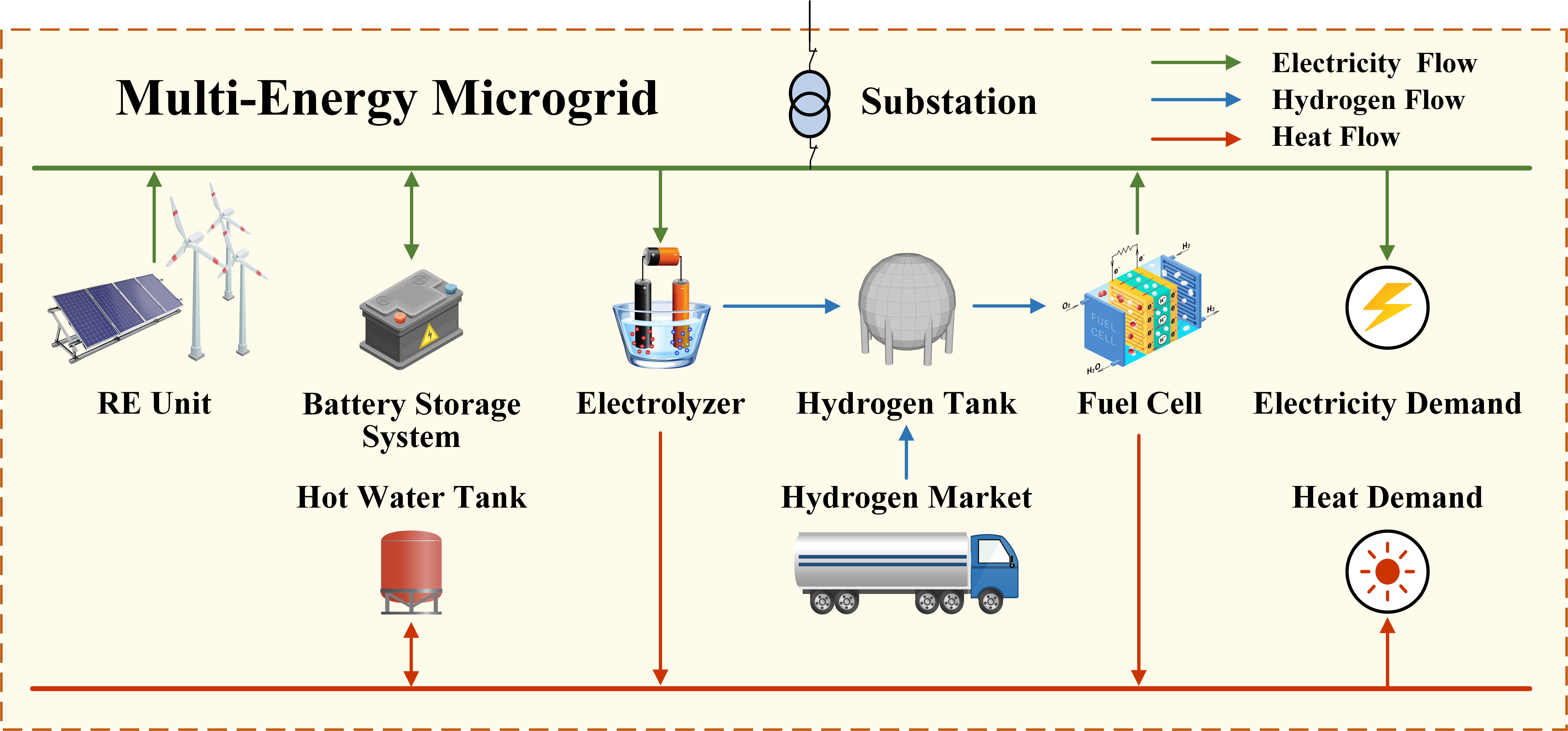}
	\vspace{-5pt}
	\caption{Typical structure of multi-energy microgrid.}
	\vspace{-15pt}
	\label{fig1}
\end{figure}

\subsection{Stage-Wise Dispatch Formulation}
\subsubsection{Pre-Dispatch Stage}
Pre-dispatch aims to minimize the day-ahead operational cost under a basic scenario from prediction. The objective function is taken as:
\setlength{\arraycolsep}{-0.1em}
\begin{eqnarray}
	&&\nonumber \textstyle\sum_{t\in[T]}\big(c_{\rm e,buy}^t p_{\rm buy}^t-c_{\rm e,sell}^tp_{\rm sell}^t\big)\Delta_t+c_{\rm g,buy}h_{\rm buy}^t\\
	&&\nonumber \quad+\;\big[\frac{1}{2}c_{\rm bss}^{\rm{deg}}\left(p_{\rm{bss,c}}^t+p_{\rm{bss,d}}^t\right)+c^{\rm om}_{\rm elz}p_{\rm elz}^{t}+c_{\rm fc}^{\rm{om}}p_{\rm fc}^{t}\big]\Delta_t\\
	&&\nonumber \quad+\left( c_{\rm elz,c}^{\rm su}y_{\rm cold}^{t}+c_{\rm elz,c}^{\rm sd}z_{\rm cold}^{t}+c_{\rm elz,w}^{\rm su}y_{\rm warm}^{t}+c_{\rm elz,w}^{\rm sd}z_{\rm warm}^{t}\right)\\
	&& \quad+ \left(c^{\rm su}_{\rm fc}y_{\rm fc}^{t}+c^{\rm sd}_{\rm fc}z_{\rm fc}^{t}\right)\label{obj_1stg}
\end{eqnarray}

\noindent In the first line, the first term includes both purchase cost and revenue from electricity transactions, and the second term accounts for hydrogen procurement expenses. The second line incorporates the degradation cost of BSS, as well as the operation and maintenance (O\&M) costs of ELZ and FC. The third and fourth lines represent, respectively, the startup and shutdown costs of ELZ and FC.

The corresponding operational constraints for pre-dispatch are shown below. \eqref{cons-1stg-1} provides RE outputs. \eqref{cons-1stg-2}--\eqref{cons-1stg-6} show the restriction on BSS's operation. Specifically, \eqref{cons-1stg-2}--\eqref{cons-1stg-4} are constraints for the charging and dischaging power of BSS, as well as its state-of-charge. \eqref{cons-1stg-5} is the dynamics equation. \eqref{cons-1stg-6} indicates that the net charging capacities of BSS should be zero after a daily charging cycle. \eqref{cons-1stg-7} and \eqref{cons-1stg-8} impose limitations on electricity transactions. 
\setlength{\arraycolsep}{-0.3em}
\begin{eqnarray}	
	&& p_{\rm wt}^t= \phi_{\rm wt}^t \overline{P}_{\rm wt},\;p_{\rm pv}^t= \phi_{\rm pv}^t \overline{P}_{\rm pv},\quad\forall t\in[T]\label{cons-1stg-1}\\
	&& 0\leq p_{{\rm bss,c}}^t\leq \overline{P}_{\rm bss}u_{\rm bss}^t,\quad \forall t\in[T]\label{cons-1stg-2}\\
	&& 0\leq p_{{\rm bss,d}}^t\leq \overline{P}_{\rm bss}\left(1-u_{\rm bss}^t\right),\quad \forall t\in[T]\label{cons-1stg-3}\\
	&& \underline{E}_{\rm bss}\leq e_{\rm bss}^t\leq \overline{E}_{\rm bss},\quad \forall t\in[T]\label{cons-1stg-4}\\
	&& e_{\rm bss}^{t+1}=e_{\rm bss}^{t}+\left(p_{{\rm bss,c}}^t\eta_{{\rm bss,c}}-p_{{\rm bss,d}}^t/\eta_{{\rm bss,d}}\right)\Delta_t,\; \forall t\in[T]\label{cons-1stg-5}\\
	&& \textstyle\sum_{t\in[T]} \left(p_{{\rm bss,c}}^t\eta_{{\rm bss,c}}-p_{{\rm bss,d}}^t/\eta_{{\rm bss,d}}\right)\Delta_t=0 \label{cons-1stg-6}\\
	&& 0\leq p_{\rm buy}^t \leq \overline{P}_{\rm sub}u_{\rm e,buy}^t,\quad \forall t\in[T]\label{cons-1stg-7}\\
	&& 0\leq p_{\rm sell}^t \leq \overline{P}_{\rm sub}\left(1-u_{\rm e,buy}^t\right),\quad \forall t\in[T]\label{cons-1stg-8}
\end{eqnarray}

\begin{figure}[]
	\centering
	\includegraphics[width=0.4\textwidth]{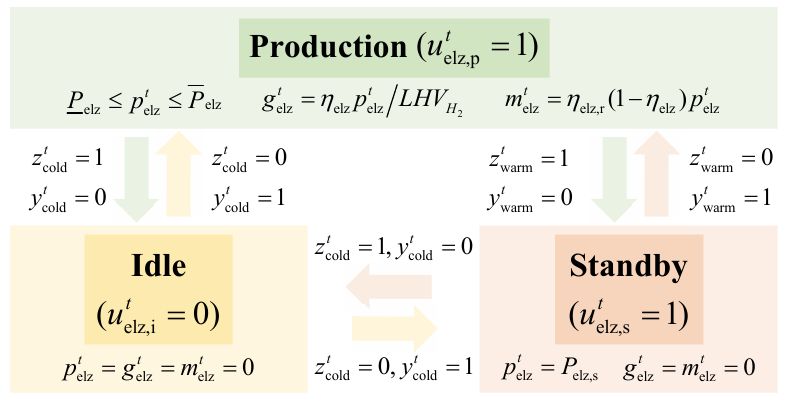}
	\vspace{-5pt}
	\caption{Three-state model for electrolyzer.}
	\label{3-state}
	\vspace{-15pt}
\end{figure}

\noindent\eqref{cons-1stg-9}--\eqref{cons-1stg-18} describe ELZ's operaion. We adopt a three-state model, including  production, standby, and idle states \cite{jia2025decentralized}, rather than the traditional two-state one that only has the production and idle, as illustrated in Fig. \ref{3-state}. Specifically, \eqref{cons-1stg-9} portrays the relationship among the mutually exclusive states. \eqref{cons-1stg-10}--\eqref{cons-1stg-11} capture ELZ's cold startup and shutdown. Of note, there is a time delay, denoted by $\tau_{\rm cold}$, for the cold startup to ensure all ELZ's components reaching the required temperature, pressure, and other working conditions from the complete turn-off (idle) state. \eqref{cons-1stg-12}--\eqref{cons-1stg-13} and \eqref{cons-1stg-14}--\eqref{cons-1stg-15} confine, respectively, the warm startup and shutdown. We mention that the warm-startup procedure is much faster comparing to the cold one, as an ELZ in the standby state maintains its necessary operating conditions with a small electric power input $P_{\rm elz,s}$. Considering the power-to-hydrogen-and-heat (P2HH) procedure, \eqref{cons-1stg-16}--\eqref{cons-1stg-16-2}, \eqref{cons-1stg-17}, and \eqref{cons-1stg-18} restrict electric power input, as well as hydrogen and heat outputs, respectively. 
\setlength{\arraycolsep}{-1.4em}
\begin{eqnarray}
	&& u_{\rm elz,p}^{t}+u_{\rm elz,s}^{t}=u_{\rm elz,i}^{t},\quad \forall t\in[T] \label{cons-1stg-9}\\
	&& u_{\rm elz,i}^{t}-u_{\rm elz,i}^{t-1}=y_{\rm cold}^{t-\tau_{\rm cold}/\Delta_t}-z_{\rm cold}^{t},\quad\forall t\in[T] \label{cons-1stg-10}\\
	&& y_{\rm cold}^{t-\tau_{\rm cold}}+z_{\rm cold}^{t}\leq 1,\quad \forall t\in[T]\label{cons-1stg-11}\\
	&& y_{\rm warm}^{t}\geq u_{\rm elz,s}^{t-1}+u_{\rm elz,p}^{t}-1,\quad\forall t\in[T]\label{cons-1stg-12}\\
	&& y_{\rm warm}^{t}\leq u_{\rm elz,s}^{t-1},\;y_{\rm warm}^{t}\leq u_{\rm elz,p}^{t},\quad\forall t\in[T]\label{cons-1stg-13}\\
	&& z_{\rm warm}^{t}\geq u_{\rm elz,s}^{t}+u_{\rm elz,p}^{t-1}-1,\quad\forall t\in[T]\label{cons-1stg-14}\\
	&& z_{\rm warm}^{t}\leq u_{\rm elz,s}^{t},\;z_{\rm warm}^{t}\leq u_{\rm elz,p}^{t-1},\quad\forall t\in[T]\label{cons-1stg-15}\\
	&& p_{\rm elz}^{t}=p_{\rm elz,p}^{t}+p_{\rm elz,s}^{t},\quad \forall t\in[T]\label{cons-1stg-16}\\
	&& \underline{P}_{\rm elz}u_{\rm elz,p}^{t}\leq p_{\rm elz,p}^{t}\leq \overline{P}_{\rm elz}u_{\rm elz,p}^{t},\quad \forall t\in[T]\label{cons-1stg-16-1}\\
	&& p_{\rm elz,s}^{t}=P_{\rm elz,s} u_{\rm elz,s}^{t},\quad \forall t\in[T]\label{cons-1stg-16-2}\\
	&& g_{\rm elz}^{t}=\eta_{\rm elz}p_{\rm elz,p}^{t}/LHV_{\rm H_2},\quad\forall t\in[T]\label{cons-1stg-17}\\
	&& m_{\rm elz}^{t}=\eta_{{\rm elz},{\rm r}}(1-\eta_{\rm elz})p_{\rm elz,p}^{t},\quad\forall t\in[T]\label{cons-1stg-18}
\end{eqnarray}
Further operational equations and constraints are given by \eqref{cons-1stg-19}--\eqref{cons-1stg-end} below. \eqref{cons-1stg-19}--\eqref{cons-1stg-20} show the restrictions on hydrogen storage, capturing the mass dynamics and capacity limits of HT. Hydrogen procurement is subject to \eqref{cons-1stg-21}--\eqref{cons-1stg-22}. As a hydrogen-driven combined-heat-and-power device, \eqref{cons-1stg-23} describes FC's hydrogen-to-power process. Also, its hydrogen and heat outputs are governed, respectively, by \eqref{cons-1stg-24} and \eqref{cons-1stg-25}.  FC's startup and shutdown are logically determined by \eqref{cons-1stg-26} and \eqref{cons-1stg-27}. Similar to \eqref{cons-1stg-19}--\eqref{cons-1stg-20}, \eqref{cons-1stg-28}--\eqref{cons-1stg-29} are constraints for HWT. Finally, \eqref{cons-1stg-30} and \eqref{cons-1stg-end} characterize heat and electric power balances, respectively.
\setlength{\arraycolsep}{-0.5em}
	\begin{eqnarray}
	&&  h_{\rm ht}^{t+1}=(1-\nu_{\rm ht})h_{\rm ht}^{t}+\left(g_{{\rm elz}}^{t}-g_{{\rm fc}}^{t}\right)\Delta_t+h_{\rm buy}^t,\; \forall t\in[T]\label{cons-1stg-19}\\
	&& 0\leq h_{\rm ht}^t\leq \overline{H}_{\rm ht},\quad \forall t\in[T]\label{cons-1stg-20}\\
	&& 0\leq h_{\rm buy}^t\leq \overline{H}_{\rm buy}u_{\rm g,buy}^t,\quad \forall t\in[T]\label{cons-1stg-21}\\
	&& \textstyle\sum_{t\in[T]}u_{\rm g,buy}^t\leq \overline{U}_{\rm g,buy}\label{cons-1stg-22}\\
	&& p_{\rm fc}^{t}=\eta_{\rm fc}LHV_{\rm H_2}g_{\rm fc}^{t},\quad \forall t\in[T]\label{cons-1stg-23}\\
	&& \underline{P}_{\rm fc}u_{\rm fc}^{t}\leq p_{\rm fc}^{t}\leq \overline{P}_{\rm fc}u_{\rm fc}^{t},\quad\forall t\in[T]\label{cons-1stg-24}\\
	&& m_{\rm fc}^{t}=\eta_{\rm fc,r}(1-\eta_{\rm fc})p_{\rm fc}^{t}/\eta_{\rm fc},\quad \forall t\in[T]\label{cons-1stg-25}\\
	&& y_{\rm fc}^{t}=\max\{u_{\rm fc}^{t}-u_{\rm fc}^{t-1},\;0\},\quad\forall t\in[T]\label{cons-1stg-26}\\
	&& z_{\rm fc}^{t}= \max\{u_{\rm fc}^{t-1}-u_{\rm fc}^{t},\;0\},\quad\forall t\in[T]\label{cons-1stg-27}\\
	&& n_{\rm hwt}^{t+1}=(1-\nu_{\rm hwt})n_{\rm hwt}^{t}+m_{\rm hwt}^t\Delta_t,\quad \forall t\in[T]\label{cons-1stg-28}\\
	&&0\leq n_{\rm hwt}^t\leq \overline{N}_{\rm hwt},\quad\forall t\in[T]\label{cons-1stg-29} \\
	&& m_{\rm elz}^t+m_{\rm fc}^t=\phi_{\rm hd}^t M_{\rm d}+m_{\rm hwt}^t,\quad \forall t\in[T] \label{cons-1stg-30} \\
	&& \nonumber p_{\rm wt}^t+p_{\rm pv}^t+\left(p_{\rm{bss,d}}^t-p_{\rm{bss,c}}^t\right)+{p_{{\rm fc}}^{t}}+\left(p_{\rm buy}^t-p_{\rm sell}^t\right)\\
	&& \quad\qquad\qquad\qquad\qquad\qquad={p_{{\rm elz}}^{t}}+\phi^t_{\rm ed} P_{\rm d},\quad \forall t\in[T]\label{cons-1stg-end}
\end{eqnarray}
\begin{remark}\label{linearizaton}
	Given that cost coefficients for binary variables $y_{\rm fc}^t$ and $z_{\rm fc}^t$ are positive in \eqref{obj_1stg}, \eqref{cons-1stg-26}--\eqref{cons-1stg-27} can be linearized to yield
	\begin{align}
	\nonumber y_{\rm fc}^{t}\geq u_{\rm fc}^{t}-u_{\rm fc}^{t-1},\;\; z_{\rm fc}^{t}\geq u_{\rm fc}^{t-1}-u_{\rm fc}^{t},\quad \forall t\in[T].
	\end{align}
\end{remark}

\subsubsection{Re-Dispatch Stage}
The re-dispatch mainly concerns the adjustment operations of MEMG's flexible resources in the future, based on some fixed pre-dispatch decisions and a certain realized operating condition. Its target is to minimize the potential intra-day operational cost:
\setlength{\arraycolsep}{-0.1em}
\begin{eqnarray}
	&&\nonumber  \textstyle\sum_{t\in[T]}\Big[(c_{\rm e,buy}^t \Delta{p}_{\rm buy}^t-c_{\rm e,sell}^t\Delta{p}_{\rm sell}^t)\\
	&&\nonumber \quad+~\frac{1}{2}c_{\rm bss}^{\rm{deg}}\!\left(\Delta{p}_{\rm{bss,c}}^t+\Delta{p}_{\rm{bss,d}}^t\right)+ c^{\rm om}_{\rm elz}\Delta{p}_{\rm elz}^{t}+ c^{\rm om}_{\rm fc}\Delta{p}_{\rm fc}^{t}\\
	&&\quad+~{\iota}_{\rm e}\tilde{p}_{\rm loss}^t+{\iota}_{\rm h}\tilde{m}_{\rm loss}^t \Big]\Delta_t.
\end{eqnarray}		
The first line gives the deviation cost of electricity transactions. The three terms in the second line represent, respectively, the degradation cost of BS, and O\&M costs of ELZ and FC, all incurred due to adjustment. As indicated in the third line, we additionally consider the compensation cost of both unmet electricity and heat demand.

The microgrid's re-dispatch constraints are given below:
\setlength{\arraycolsep}{-0.4em}
\begin{eqnarray}	
	&& \Delta p_{\rm bs,c}^t=\tilde{p}_{\rm bs,c}^t-p_{\rm bs,c}^t,\;\Delta p_{\rm bs,d}^t=\tilde{p}_{\rm bs,d}^t-p_{\rm bs,d}^t,\; \forall t\in[T]\label{cons-2stg-1} \\
	&& \Delta p_{\rm elz}^{t}=\tilde{p}_{\rm elz}^{t}-p_{\rm elz}^{t},\quad \forall t\in[T]\label{cons-2stg-2} \\	
	&& \Delta p_{\rm fc}^{t}=\tilde{p}_{\rm fc}^{t}-p_{\rm fc}^{t},\quad \forall t\in[T]\label{cons-2stg-3} \\
	&& \Delta p_{\rm sell}^t=\tilde{p}_{\rm sell}^t-p_{\rm sell}^t,\;\Delta p_{\rm buy}^t=\tilde{p}_{\rm buy}^t-p_{\rm buy}^t,\quad \forall t\in[T]\label{cons-2stg-4} \\
	&& 0\leq \tilde{p}_{\rm wt}^t\leq \tilde{\phi}_{\rm wt}^t P_{\rm wt},\;0\leq \tilde{p}_{\rm pv}^t\leq \tilde{\phi}_{\rm pv}^t P_{\rm pv},\quad\forall t\in[T]\label{cons-2stg-5}\\
	&& 0\leq \tilde{p}_{{\rm bss,c}}^t,\;\tilde{p}_{{\rm bss,d}}^t\leq P_{\rm bss},\quad \forall t\in[T]\label{cons-2stg-6}\\
	&& \tilde{e}_{\rm bss}^{t+1}=\tilde{e}_{\rm bss}^{t}\!+\!\left(\tilde{p}_{{\rm bss,c}}^t\eta_{{\rm bss,c}}-\tilde{p}_{{\rm bss,d}}^t/\eta_{{\rm bss,d}}\right)\Delta_t,\; \forall t\in[T]\label{cons-2stg-7}\\
	&& \underline{E}_{\rm bss}\leq \tilde{e}_{\rm bss}^t\leq E_{\rm bss},\quad \forall t\in[T]\label{cons-2stg-8}\\
	&& \textstyle\sum_{t\in[T]} \left(\tilde{p}_{{\rm bss,c}}^t\eta_{{\rm bss,c}}-\tilde{p}_{{\rm bss,d}}^t/\eta_{{\rm bss,d}}\right)\Delta_t=0\label{cons-2stg-9} \\
	&& p_{\rm buy}^t(1-{\delta}_{\rm buy})\leq \tilde{p}_{\rm buy}^t\leq p_{\rm buy}^t(1+{\delta}_{\rm buy}),\quad \forall t\in[T]\label{cons-2stg-10}\\
	&& p_{\rm sell}^t(1-{\delta}_{\rm sell})\leq \tilde{p}_{\rm sell}^t\leq p_{\rm sell}^t(1+{\delta}_{\rm sell}),\quad \forall t\in[T]\label{cons-2stg-11}\\
	&& 0\leq \tilde{p}_{\rm buy}^t,\;\tilde{p}_{\rm sell}^t\leq \overline{P}_{\rm sub},\quad \forall t\in[T]\label{cons-2stg-12}\\
	&& \tilde{p}_{\rm elz}^{t}=\tilde{p}_{\rm elz,p}^{t}+\tilde{p}_{\rm elz,s}^{t},\quad \forall t\in[T]\label{cons-2stg-13}\\
	&& \underline{P}_{\rm elz}u_{\rm elz,p}^{t}\leq\tilde{p}_{\rm elz,p}^{t}\leq\overline{P}_{\rm elz}u_{\rm elz,p}^{t},\quad \forall t\in[T]\\
	&& \tilde{p}_{\rm elz,s}^{t}=P_{\rm elz,s} u_{\rm elz,s}^{t},\quad \forall t\in[T]\\
	&& \tilde{g}_{\rm elz}^{t}=\eta_{\rm elz}\tilde{p}_{\rm elz,p}^{t}/LHV_{\rm H_2},\quad \forall t\in[T]\label{cons-2stg-14}\\
	&& \tilde{m}_{\rm elz}^{t}=\eta_{{\rm elz},{\rm r}}(1-\eta_{\rm elz})\tilde{p}_{\rm elz,p}^{t},\quad\forall t\in[T]\label{cons-2stg-15}\\
	&& \tilde{h}_{\rm ht}^{t+1}=(1-\nu_{\rm ht})\tilde{h}_{\rm ht}^{t}+\left(\tilde{g}_{{\rm elz}}^{t}-\tilde{g}_{{\rm fc}}^{t}\right)\Delta_t+h_{\rm buy}^t,\; \forall t\in[T]\label{cons-2stg-16}\\
	&& 0\leq \tilde{h}_{\rm ht}^t\leq \overline{H}_{\rm ht},\quad \forall t\in[T]\label{cons-2stg-17}\\
	&& \tilde{p}_{\rm fc}^{t}=\eta_{\rm fc}LHV_{\rm H_2}\tilde{g}_{\rm fc}^{t},\quad \forall t\in[T]\label{cons-2stg-18}\\
	&& \underline{P}_{\rm fc}u_{\rm fc}^{t}\leq \tilde{p}_{\rm fc}^{t}\leq \overline{P}_{\rm fc}u_{\rm fc}^{t},\quad \forall t\in[T]\label{cons-2stg-21}\\
	&& \tilde{m}_{\rm fc}^{t}=\eta_{\rm fc,r}(1-\eta_{\rm fc})\tilde{p}_{\rm fc}^{t}/\eta_{\rm fc},\quad \forall t\in[T]\label{cons-2stg-22}\\
	&& \tilde{n}_{\rm hwt}^{t+1}=(1-\nu_{\rm hwt})\tilde{n}_{\rm hwt}^{t}+\tilde{m}_{\rm hwt}^t\Delta_t,\quad \forall t\in[T]\label{cons-2stg-23}\\
	&&0\leq \tilde{n}_{\rm hwt}^t\leq \overline{N}_{\rm hwt},\quad\forall t\in[T] \label{cons-2stg-24}\\
	&& \tilde{m}_{\rm elz}^t+\tilde{m}_{\rm fc}^t=\left(\tilde{\phi}_{\rm hd}^t{M}_{\rm d}-\tilde{m}_{\rm loss}^t\right)+\tilde{m}_{\rm hwt}^t,\quad \forall t\in[T]\label{cons-2stg-24-1} \\
	&& 0\leq \tilde{m}_{\rm loss}^t\leq \tilde{\phi}_{\rm hd}^t{M}_{\rm d},\quad \forall t\in[T]\label{cons-2stg-25}\\
	&& \nonumber \tilde{p}_{\rm wt}^t+\tilde{p}_{\rm pv}^t+\left(\tilde{p}_{\rm{bs,d}}^t-\tilde{p}_{\rm{bs,c}}^t\right)+{\tilde{p}_{{\rm fc}}^{t}}+\left(\tilde{p}_{\rm buy}^t-\tilde{p}_{\rm sell}^t\right)\\
	&& \qquad\qquad\qquad\qquad={\tilde{p}_{{\rm elz}}^{t}}+(\tilde{\phi}^t_{\rm ed}{P}_{\rm d}-\tilde{p}_{\rm loss}^t),\quad \forall t\in[T]\label{cons-2stg-26} \\
	&& 0\leq \tilde{p}_{\rm loss}^t\leq \tilde{\phi}^t_{\rm ed}{P}_{\rm d},\quad \forall t\in[T].\label{cons-2stg-end}
\end{eqnarray}
In the above,
\eqref{cons-2stg-1}--\eqref{cons-2stg-4} define some deviations between re- and pre-dispatch decisions. Similar to the pre-dispatch problem, \eqref{cons-2stg-5}, \eqref{cons-2stg-6}--\eqref{cons-2stg-9}, \eqref{cons-2stg-13}--\eqref{cons-2stg-15},  \eqref{cons-2stg-16}--\eqref{cons-2stg-17}, \eqref{cons-2stg-18}--\eqref{cons-2stg-22}, and \eqref{cons-2stg-23}--\eqref{cons-2stg-24} are  constraints for RE units, BSS, ELZ, HT, FC, and HWT, respectively, as re-dispatch counterparts. Note that RE curtailment is allowed in \eqref{cons-2stg-5}, and we equivalently do not include the state variables for BSS in \eqref{cons-2stg-6} aligned with the ideas in \cite{10582313}. \eqref{cons-2stg-10}--\eqref{cons-2stg-12} confine the electricity transactions in the intra-day market, allowing some adjustments on top of the day-ahead contract. By introducing $\tilde{m}_{\rm loss}^t$ in \eqref{cons-2stg-25} and $\tilde{p}_{\rm loss}^t$ in \eqref{cons-2stg-end} to quantify unmet heat and electricity demand, respectively, \eqref{cons-2stg-24-1} and \eqref{cons-2stg-26} individually express the heat and electric power balance.

\subsection{Integrated Formulation based on DRO}

We integrate the aforementioned two problems into a single mathematical program, using the re-dispatch one to evaluate the quality of the pre-dispatch decisions. Nevertheless, the re-dispatch stage confronts multiple uncertainties, including intermittent RE outputs (i.e., $\tilde{\phi}_{\rm wt}^t$ and $\tilde{\phi}_{\rm pv}^t$) as well as random electricity and heat demand (i.e., $\tilde{\phi}_{\rm ed}^t$ and $\tilde{\phi}_{\rm hd}^t$). To hedge against these operational risks, the two-stage DRO approach is adopted, resulting in the following \textbf{DRD-MEMG} problem. The goal of \textbf{DRD-MEMG} is to determine the optimal pre-dispatch strategy to minimize the sum of pre-dispatch cost and the worst-case expectation of re-dispatch cost. For ease of exposition, we express it in the compact form below. 
\begin{align}
\hspace{-2mm}
{\textbf{DRD-MEMG}}:~	\mathit{w} =&\; \min_{\mathbf{x}} ~\mathbf{c}^{\intercal}\mathbf{x}+\max_{\mathbb{P}\in\mathcal{P}}~\mathbb{E}_{\mathbb{P}}[Q(\mathbf{x},\bm{\xi})]  \label{1stg-1} \\ 
	\mathrm{s.t.} ~ &
	\mathbf{A}\mathbf{x}\leq \mathbf{b},\quad \mathbf{x}\in\left\lbrace 0,1\right\rbrace ^{n_{\rm x}}\times\mathbb{R}_{+}^{m_{\rm x}} \label{1stg-3}
\end{align}
In the above, $\mathbf{x}$ is the first-stage decision vector that represents the here-and-now pre-dispatch decisions. \eqref{1stg-3} corresponds to constraints \eqref{cons-1stg-1}--\eqref{cons-1stg-end}. By collecting all uncertain parameters, random vector $\bm{\xi}=(\xi^1,\cdots,\xi^{m_{\xi}})$ is defined on a measurable space $(\Xi,\mathcal{F})$ with an (unknown) probability distribution $\mathbb{P}$. $\Xi\triangleq\times_{i=1}^{m_{\xi}}[\underline{\xi}^i,\overline{\xi}^i]\subseteq\mathbb{R}^{m_{\xi}}$ is a box-type \textit{sample space} that can be obtained from historical data or expert knowledge, and $\mathcal{F}$ is a $\sigma$-algebra of $\Xi$. Further, we define $\overline{\bm{\xi}}\triangleq(\overline{\xi}^1,\cdots,\overline{\xi}^{m_{\xi}})$ and $\underline{\bm{\xi}}\triangleq(\underline{\xi}^1,\cdots,\underline{\xi}^{m_{\xi}})$. 
$\mathcal{P}$ denotes the ambiguity set. $Q(\mathbf{x},\bm{\xi})$ is the recourse value given $\mathbf{x}$ and $\bm{\xi}$:
\begin{align}
	Q(\mathbf{x},\bm{\xi})=\min_{\mathbf{y}}&~\mathbf{d}^{\intercal}\mathbf{y}\label{2stage-1}\\
	{\rm s.t.}&~\mathbf{F}\mathbf{y}\leq\mathbf{h}-\mathbf{G}\mathbf{x}-\mathbf{K}\bm{\xi},\quad\mathbf{y}\in\mathbb{R}_{+}^{m_{\rm y}},\label{2stage-3}
\end{align}
where $\mathbf{y}$ denotes the wait-and-see decision vector in the second-stage for HM's re-dispatch, and \eqref{2stage-3} corresponds to constraints \eqref{cons-2stg-1}--\eqref{cons-2stg-end}. Finally, we use $\mathbb{E}_{\mathbb{P}}[Q(\mathbf{x},\bm{\xi})]$ to denote the expectation of $Q(\mathbf{x},\bm{\xi})$ with respect to the distribution $\mathbb{P}$. \vspace{-5pt}

\begin{remark}
	Due to the inclusion of $\tilde{p}_{\rm loss}^t$ and $\tilde{m}_{\rm loss}^t$, which mathematically serve as non-negative auxiliary variables, \textbf{DRD-MEMG} has \textit{relatively complete recourse}, i.e., for all $\mathbf{x}$ satisfying \eqref{1stg-3}, and for all $\bm{\xi}\in\Xi$, the recourse problem is always feasible and  $|Q(\mathbf{x},\bm{\xi})|<\infty$. \vspace{-5pt}
\end{remark}

Given $S$ (training) samples $\{\bm{\xi}^{\rm e}_s\}_{s=1}^S$ of $\bm{\xi}$, we consider the Wasserstein ambiguity set for data-driven decision-making:
\begin{align}
	\mathcal{P}=\Big\{\mathbb{P}\in\mathcal{M}\left(\Xi,\mathcal{F}\right)\;:\;\mathfrak{d}_{\rm w}\left({\mathbb{P}}, \mathbb{P}^{\rm e}\right)\leq r\Big\}.\label{ambiguity-w}
\end{align}
Here, $\mathcal{M}(\Xi,\mathcal{F})$ is the set of all probability distributions on the measurable space $(\Xi,\mathcal{F})$. Further, $\mathbb{P}^{\rm e}\triangleq\sum_{s=1}^S\pi_s^{\rm e}\delta_{{\bm{\xi}}_s^{\rm e}}$ denotes the empirical distribution on the samples, where $\delta_{{\bm{\xi}}_s^{\rm e}}$ is a Dirac measure concentrating unit mass at ${\bm{\xi}}_s^{\rm e}$, and $\pi_s^{\rm e}$ is ${\bm{\xi}}_s^{\rm e}$'s probability satisfying $\sum_{s=1}^S\pi_s^{\rm e}=1$. $\mathfrak{d}_{\rm w}\left({\mathbb{P}}, \mathbb{P}^{\rm e}\right)$ is the Wasserstein metric between distributions $\mathbb{P}$ and $\mathbb{P}^{\rm e}$:
\begin{align}
	\hspace{-1mm}
	\mathfrak{d}_{\rm w}\left({\mathbb{P}}, \mathbb{P}^{\rm e}\right)&=\min_{\Pi\in\mathcal{S}(\mathbb{P},\mathbb{P}^{\rm e})}\textstyle\iint_{\Xi\times\Xi}\left\|\bm{\xi}-\bm{\zeta} \right\| \Pi(\mathrm{d}\bm{\xi},\mathrm{d}\bm{\zeta})\label{w-metric-1}\\
	&=\min_{\mathbb{P}_s\in\mathcal{M}(\Xi,\mathcal{F})}\textstyle\sum_{s=1}^S\pi_s^{\rm e}\int_{\Xi}\left\|\bm{\xi}-\bm{\xi}_s^{\rm e} \right\|\mathbb{P}_s(\mathrm{d}\bm{\xi}).\label{w-metric-2}
\end{align}
In \eqref{w-metric-1}, $\mathcal{S}(\mathbb{P},\mathbb{P}^{\rm e})$ denotes the set of all joint distributions of $\bm{\xi}$ and $\bm{\zeta}$ with marginals $\mathbb{P}$ and $\mathbb{P}^{\rm e}$, respectively; $\left\|\cdot \right\|$ is an arbitrary norm on $\mathbb{R}^{m_{\rm \xi}}$, which is chosen as the $L_1$ norm in our work, i.e., $\left\|\bm{\xi}-\bm{\zeta} \right\|_1=\sum_{i=1}^{m_{\xi}}|\xi_i-\zeta_i|$. \eqref{w-metric-2} is obtained by virtue of $\bm{\xi}$'s conditional distribution $\mathbb{P}_s$ given $\bm{\zeta}=\bm{\xi}^{\rm e}_s$. Obiviously, $\mathbb{P}=\sum_{s=1}^S\pi_s\mathbb{P}_s$. The ambiguity set $\mathcal{P}$ can be viewed as a Wasserstein ball centered at the empirical distribution $\mathbb{P}^{\rm e}$ \cite{mohajerin2018data}, where the pre-defined parameter  $r\geq 0$ is its \textit{radius}. We mention that $r$ is also referred to as the level of robustness \cite{rahimian2022frameworks}, and the result below can be easily derived. \vspace{-5pt}
\begin{prop}\label{prop-r}
	For $r_1<r_2$, i.e., $\mathcal{P}_1\subseteq\mathcal{P}_2$, we have 
	$w_1\leq w_2$.\vspace{-5pt}
\end{prop}

With the notations introduced above, we further define the worst-case expectation problem (\textbf{WCEP}) for a given $\mathbf{x}$ as follows:
\setlength{\arraycolsep}{-0.5em}
\begin{eqnarray}
	&&{\textbf{WCEP}}:\;v(\mathbf{x})\!=\!\max_{\mathbb{P}\in\mathcal{P}}\mathbb{E}_{\mathbb{P}}\left[Q(\mathbf{x},\bm{\xi})\right]\!=\!\max_{\mathbb{P}\in\mathcal{P}}\!\textstyle\int_{\Xi}\!Q(\mathbf{x},\bm{\xi})\mathbb{P}(\mathrm{d}\bm{\xi})\label{wcep-1}\\
	&&\quad\quad\;\;\;=\max_{\mathbb{P}_s\in\mathcal{M}(\Xi,\mathcal{F})}\textstyle\sum_{s=1}^S\pi_s^{\rm e}\int_{\Xi}Q(\mathbf{x},\bm{\xi})\mathbb{P}_s(\mathrm{d}\bm{\xi})\label{wcep-2}\\
	&&\quad\quad\quad\quad\quad\;\;\;{\rm s.t.}\quad\textstyle\sum_{s=1}^S\pi_s^{\rm e}\int_{\Xi}\left\|\bm{\xi}-\bm{\xi}_s^{\rm e} \right\|_1\mathbb{P}_s(\mathrm{d}\bm{\xi})\leq r\label{wcep-3}\\
	&&\quad\quad\;\;\;=\max_{\mathbb{P}_s\in\mathcal{M}(\Xi,\mathcal{F})}{\textstyle{\sum_{s=1}^S}}\textstyle\pi_s^{\rm e}\int_{\Xi}Q(\mathbf{x},\bm{\xi}_s)\mathbb{P}_s(\mathrm{d}\bm{\xi}_s)\label{wcep-4}\\
	&&\quad\quad\quad\quad\quad\;\;\;{\rm s.t.}\quad\textstyle\sum_{s=1}^S\pi_s^{\rm e}\int_{\Xi}\left\|\bm{\xi}_s-\bm{\xi}_s^{\rm e} \right\|_1\mathbb{P}_s(\mathrm{d}\bm{\xi}_s)\leq r.\label{wcep-5}
\end{eqnarray}
The fourth equality above is obtained by introducing $S$ independent replicas of $\bm{\xi}$, denoted by $\{\bm{\xi}_s\}_{s=1}^S$. To be explained in the next section, this transformation will help us accelerate the solution procedure. Note that \textbf{WCEP} in the form of \eqref{wcep-4}--\eqref{wcep-5} can be viewed as an infinite-dimensional linear program (LP). Following the result in \cite{lu2024}, we have the theorem below.\vspace{-5pt}

\begin{thm}\label{eq-summation}
	(Adapted from Corollary 4 in \cite{lu2024}) 
	For each $s\in[S]$, let $\Big\{\{\bm{\xi}_{sn}\}_{n=1}^{N_s},\{\pi_{sn}\}_{n=1}^{N_s}\Big\}$ denote a discrete distribution with scenarios $\{\bm{\xi}_{sn}\}_{n=1}^{N_s}$, and  $\{\pi_{sn}\}_{n=1}^{N_s}$ being the associated probabilities. Then, there exist S sequences of  $\{\bm{\xi}_{sn}\}_{n=1}^{N_s}$'s such that the following equivalence holds: 
	\setlength{\arraycolsep}{-0.5em}
	\begin{eqnarray}
		&&\nonumber v(\mathbf{x})=\lim_{\substack{N_s\rightarrow+\infty\\s\in[S]}}\max_{\substack{\{\{\bm{\xi}_{sn}\}_{n=1}^{N_s}}\}_{s=1}^S}\max_{\substack{\{\{\pi_{sn}\}_{n=1}^{N_s}}\}_{s=1}^S}\\
		&&\quad\qquad\qquad\qquad\qquad\textstyle\sum_{s=1}^S\sum_{n=1}^{N_s}\pi_s^{\rm e}Q(\mathbf{x},\bm{\xi}_{sn})\pi_{sn}\label{wcep-6}\\
		&&{\rm s.t.} ~ \textstyle \sum_{n=1}^{N_s}\pi_{sn}=1,\quad \forall s\in[S] \label{wcep-7}\\
		&&\quad\quad{\textstyle\sum_{s=1}^S} \textstyle\sum_{n=1}^{N_s}\pi_s^{\rm e}\left\|\bm{\xi}_{sn}-\bm{\xi}_s^{\rm e} \right\|_1\pi_{sn}\leq r
		\label{wcep-8}\\
		&&\quad\quad\bm{\xi}_{sn}\in\Xi,\;\;\pi_{sn}\in\mathbb{R}_+,\quad \forall n\in[N_s], \forall s\in[S].\label{wcep-9}
	\end{eqnarray}
\end{thm}

Theorem \ref{eq-summation} says that \textbf{WCEP} in the integration-based form \eqref{wcep-4}--\eqref{wcep-5} can be equivalently replaced by a weighted sum over a series of discrete scenarios. It is worth highlighting that this conversion enables us to approach \textbf{WCEP}'s optimal solution by gradually expanding  $\Big\{\{\bm{\xi}_{sn}\}_{n=1}^{N_s},\{\pi_{sn}\}_{n=1}^{N_s}\Big\}$'s. Also, \eqref{wcep-6} implies that we can navigate scenarios and their associated probabilities individually. Inspired by these insights, in the next section, we will showcase that \textbf{WCEP} can be addressed in a finite number of steps by leveraging the well-known CG algorithm \cite{lubbecke2005selected}, which is suitable for solving LPs with a vast number of columns. Thus, the number of resulting discrete scenarios also remains finite. Further, an exact and efficient algorithm, with CG subroutine for \textbf{WCEP}, is tailored in Section \ref{ccg} for the solution of the complete \textbf{DRD-MEMG}.\vspace{-5pt}

\begin{remark}\label{sparsity}
	i) The equivalence of \textbf{WCEP} in Theorem \ref{eq-summation} is derived from the \textit{primal perspective}, which contrasts with the mainstream research that primarily adopts the dual perspective, e.g., \cite{mohajerin2018data,zhao2018data}. It is worth highlighting that this equivalence is rather intuitive and facilitates the design of a more efficient algorithm for solving \textbf{DRD-MEMG}. 
	
	ii) In fact, according to \cite{stochastic1,yue2022linear,gao2023distributionally},  \textbf{WCEP} can be further reformulated as a finite mathematical program defined on a discrete probability distribution that is supported on at most $S+1$ scenarios\cite{lu2024}.  
\end{remark}

\section{Finite-Step Solution of WCEP via CG}\label{wcep-cg}

Consider a fixed $\hat{\mathbf{x}}$. For each $s\in[S]$, include some given scenarios in a set $\Omega_s(\hat{\mathbf{x}})$. By restricting \textbf{WCEP} on these scenarios, CG iteratively begins with a pricing master problem (\textbf{PMP}), which is an LP that determines the worst-case probability distribution over $\Omega_s(\hat{\mathbf{x}})$'s.
\begin{eqnarray}
	&&{\textbf{PMP}}:\underline{v}(\hat{\mathbf{x}}) =\max_{\pi_{\bm{\xi}_s}}~\textstyle\sum_{s=1}^S \sum_{\bm{\xi}_s\in\Omega_s(\hat{\mathbf{x}})}\pi_s^{\rm e}Q(\hat{\mathbf{x}},\bm{\xi}_{s})\pi_{\bm{\xi}_s}\label{pmp-1} \\ 
	&&\mathrm{s.t.}~ \textstyle\sum_{\bm{\xi}_s\in\Omega_s(\hat{\mathbf{x}})}\pi_{\bm{\xi}_s}=1,\quad \forall s\in[S]:\quad (\alpha_s\in\mathbb{R})\label{pmp-2} \\
	&&\quad\;\;\;\textstyle\sum_{s=1}^S\pi_s^{\rm e}\sum_{\bm{\xi}_s\in\Omega_s(\hat{\mathbf{x}})}\left\|\bm{\xi}_{s}-\bm{\xi}_s^{\rm e} \right\|_1\pi_{\bm{\xi}_s}\leq r:\; (\beta\in\mathbb{R}_+)\label{pmp-3} \\
	&&\quad\;\;\;\pi_{\bm{\xi}_s}\in\mathbb{R}_{+},\quad \forall \bm{\xi}_s\in\Omega_s(\hat{\mathbf{x}}), \forall s\in[S]\label{pmp-4}
\end{eqnarray}
Note that subscript $\bm{\xi}_s$ is used to make a distinction between decision vectors associated with different scenarios $\bm{\xi}_s\in\Omega_s(\hat{\mathbf{x}})$.

Collecting the optimal multipliers (shadow prices) $\{\hat{\alpha}_s\}_{s=1}^S$ and $\hat{\beta}$ from \textbf{PMP}, a pricing subproblem (\textbf{PSP}$_s$) is defined for each $s\in[S]$, which seeks for the optimal scenario corresponding to the maximized reduced cost, as follows:
\setlength{\arraycolsep}{-0.5em}
\begin{eqnarray}
	&&{\textbf{PSP}_s}:~{\mu}_s(\hat{\mathbf{x}}) = \max_{\bm{\xi}_s\in\Xi} ~\pi_s^{\rm e}(Q(\hat{\mathbf{x}},\bm{\xi}_s)\!-\!\hat{\beta}\left\|\bm{\xi}_s\!-\!\bm{\xi}_s^{\rm e}\right\|_1)\!-\!\hat{\alpha}_s\label{psp-1}\\
	&&\quad\quad\;=\max_{\bm{\xi}_s\in\Xi} ~\pi_s^{\rm e}(\min_{\mathbf{y}_s\in\mathbb{R}_{+}^{m_{\rm y}}}\mathbf{d}^{\intercal}\mathbf{y}_s-\hat{\beta}\left\|\bm{\xi}_s-\bm{\xi}_s^{\rm e}\right\|_1)-\hat{\alpha}_s\label{psp-2}\\
	&&\quad\quad\quad\quad{\rm s.t.}~\mathbf{F}\mathbf{y}_s\leq\mathbf{h}-\mathbf{G}\hat{\mathbf{x}}-\mathbf{K}\bm{\xi}_s:\quad(\bm{\lambda}_s\in\mathbb{R}_{+}^{m_{\lambda}})\label{psp-3}.
\end{eqnarray}
\textbf{PSP}$_s$ in  \eqref{psp-2}--\eqref{psp-3} is a non-linear bilevel optimization (BLO) problem, which possesses heavy computational burden. To address this issue, we first propose the following result, showing that the optimal value of \textbf{PSP}$_s$ can be achieved over a few simple scenarios. \vspace{-5pt}

\begin{prop}\label{prop-psp}
	There exists an optimal $\hat{\bm{\xi}}_{s}=(\hat{\xi}_{s}^1,\cdots,\hat{\xi}^{m_{\xi}}_{s})$ to \textbf{PSP}$_s$ satisfying $\hat{\xi}_{s}^i\in\{\overline{\xi}^{i},\underline{\xi}^{i},\xi_{s}^{{\rm e},i}\}$ for all $i\in[m_{\xi}]$.\vspace{-5pt}
\end{prop}
\begin{proof}
	By dualizing the lower-level minimization problem, and by linearizing the $L_1$ norm term, \textbf{PSP}$_s$ becomes eqivalent to:
	\begin{align}
		\hspace{-4mm}
		\max_{\bm{\xi}_s,\bm{\lambda}_s,\mathbf{z}_s}~&\pi_s^{\rm e}[(\mathbf{G}\hat{\mathbf{x}}+\mathbf{K}\bm{\xi}_s-\mathbf{h})^{\intercal}\bm{\lambda}_s-\hat{\beta}\mathbf{1}^{\intercal}\mathbf{z}_s]-\hat{\alpha}_s\label{lower-dual-1}\\
		{\rm s.t.}~~~&\mathbf{d}+\mathbf{F}^{\intercal}\bm{\lambda}_s\geq \mathbf{0},\;\; \bm{\lambda}_s\in\mathbb{R}_{+}^{m_{\lambda}}\label{lower-dual-2}\\
		& \underline{\bm{\xi}}\leq\bm{\xi}_s\leq\overline{\bm{\xi}},\;\;\mathbf{z}_s\geq\bm{\xi}_s-\bm{\xi}_s^{\rm e},\;\;\mathbf{z}_s\geq\bm{\xi}_s^{\rm e}-\bm{\xi}_s. \label{lower-dual-5}
	\end{align}
	Note that problem \eqref{lower-dual-1}--\eqref{lower-dual-5} is a bilinear program that is linear in $(\bm{\xi}_s,\mathbf z_s)$ and $\bm \lambda_s$ respectively. Hence, for any fixed $\bm\lambda_s$, and following the Fundamental Theorem  of linear programming, there must exist an optimal $(\hat{\bm{\xi}}_s,\hat{\bm{z}}_s)$ that is one of the extreme points of polyhedron $\widetilde{\Xi}_s\triangleq\{(\bm{\xi}_s,\mathbf{z}_s)\in\mathbb{R}^{2m_{\xi}}:\eqref{lower-dual-5}\}$. 
	
	We note that $\widetilde{\Xi}_s$ restricts $(\bm{\xi}_s,\mathbf{z}_s)$ by imposing bounds on its individual components. For the $i$-th component, denoted by $({\xi}_s^i,{z}_s^i)$, its feasible set is 
	\begin{align}
		\nonumber\widetilde{\Xi}_s^i&=\big\{(\xi_s^i,z_s^i)\in\mathbb{R}^2:~\underline{\xi}^i\leq\xi_s^i\leq\overline{\xi}^i,\\
		&\quad\quad\quad\quad\quad z_s^i\geq \xi_s^i-\xi_s^{{\rm e},i},\;\;z_s^i\geq \xi_s^{{\rm e},i}-\xi_s^i\big\}.
	\end{align}
   Clearly, $\widetilde{\Xi}_s^i$ is a simple two-dimensional polyhedron and  $\widetilde{\Xi}_s=\times_{i=1}^{m_{\xi}}\widetilde{\Xi}_s^i$. According to Proposition 2.1.4 in \cite{bertsekas2009convex}, every extreme point of $\widetilde{\Xi}_s^i$ can be derived by converting any two constraints into equations and computing their intersections. By evaluating all $\binom{4}{2}=6$ possibilities and removing infeasible ones, we note that $\widetilde{\Xi}_s^i$ has three extreme points, i.e., $(\overline{\xi}^i,\overline{\xi}^i-\xi^{{\rm e},i})$, $(\underline{\xi}^i,\xi^{{\rm e},i}-\underline{\xi}^i)$, and $(\xi^{{\rm e},i},0)$, which completes the proof.  
\end{proof}\vspace{-5pt}

One major advantage of Proposition \ref{prop-psp} is that we can reduce $\Xi$ to a set of discrete scenarios to mitigate the computational burden of \textbf{PSP}$_s$. Specifically, let $\Delta^+_s\triangleq\overline{\bm{\xi}}-\bm{\xi}^{\rm e}_s$ and $\Delta_{s}^-\triangleq\bm{\xi}^{\rm e}_s-\underline{\bm{\xi}}$, and introduce binary vectors $\mathbf{u}_s^+,\mathbf{u}_s^-\in\{0,1\}^{m_{\xi}}$. Thus, an optimal $\bm{\xi}_s$ to \textbf{PSP}$_s$ can be rewritten in the form of $\bm{\xi}_s=\bm{\xi}_s^{\rm e}+\Delta^+_s\circ\mathbf{u}_s^+-\Delta^-_s\circ\mathbf{u}_s^-$ subject to $\mathbf{u}_s^++\mathbf{u}_s^-\leq\mathbf{1}$. Consequently, through its equivalence in \eqref{lower-dual-1}--\eqref{lower-dual-5} and the ``big-M'' linearization method \cite{9416873}, \textbf{PSP}$_s$ can be reduced to a  mixed-integer linear program (MILP): \vspace{-5pt}
\begin{cor}\label{cor-PSP}
	\textbf{PSP}$_s$ is equivalent to the following MILP:
	\begin{align}
		\hspace{-7mm}
		\nonumber{\mu}_s(\hat{\mathbf{x}}&)=\max_{\bm{\lambda}_s,\mathbf{z}_s,\mathbf{u}_s^+,\mathbf{u}_s^-,\mathbf{t}_s^+,\mathbf{t}_s^-}~\pi_s^{\rm e}[(\mathbf{G}\hat{\mathbf{x}}+\mathbf{K}\bm{\xi}_s^{\rm e}-\mathbf{h})^{\intercal}\bm{\lambda}_s\\
		&\quad\;\;\;-\hat{\beta}\mathbf{1}^{\intercal}\mathbf{z}_s +\mathbf{K}((\Delta_s^+)^{\intercal}\mathbf{t}_s^+-(\Delta_s^-)^{\intercal}\mathbf{t}_s^-)   ]-\hat{\alpha}_s\label{psp-4}\\
		{\rm s.t.}~&\mathbf{d}+\mathbf{F}^{\intercal}\bm{\lambda}_s\geq \mathbf{0}\label{psp-5}\\
		& \mathbf{z}_s=\Delta^+_s\circ\mathbf{u}_s^++\Delta^-_s\circ\mathbf{u}_s^-\label{psp-6}\\
		&\mathbf{t}_s^k\leq\bm{\lambda}_s,\; \mathbf{t}_s^k\geq\bm{\lambda}_s-M(\mathbf{1}-\mathbf{u}_s^k), \;\; \forall k\in\{+,-\}\label{psp-7}\\
		&\mathbf{0}\leq\mathbf{t}_s^k\leq M\mathbf{u}_s^k,\quad \forall k\in\{+,-\}\label{psp-8}\\
		& \mathbf{u}_s^++\mathbf{u}_s^-\leq\mathbf{1}\label{psp-9}\\
		& \bm{\lambda}_s\in\mathbb{R}_{+}^{m_{\lambda}},\;\mathbf{u}_s^+\in\{0,1\}^{m_{\xi}},\;\mathbf{u}_s^-\in\{0,1\}^{m_{\xi}},\label{psp-10}
	\end{align}
	where $M$ is a sufficiently large number. \vspace{-5pt}
\end{cor}
\noindent We note that \eqref{psp-4}--\eqref{psp-10} is computationally much more friendly than \eqref{psp-2}--\eqref{psp-3}, and can be directly handled by any MILP solver. Therefore, we adopt it as \textbf{PSP}$_s$ in CG.

After solving \textbf{PSP}$_s$'s, whenever there exists some $s'$ such that its corresponding reduced cost $\mu_{s'}(\hat{\mathbf{x}})>0$, create a new variable $\pi_{\hat{\bm{\xi}}_{s'}}$ to \textbf{PMP} and update $\Omega_{s'}(\hat{\mathbf{x}})\leftarrow\Omega_{s'}(\hat{\mathbf{x}})\cup\{\hat{\bm{\xi}}_{s'}\}$. Then, repeat the solution procedure until $\mu_{s}(\hat{\mathbf{x}})\leq0$ is satisfied for all $s\in[S]$. 

According to Proposition \ref{prop-psp}, the following result regarding the finite-step convergence of the aforementioned CG procedure can be readily established. The detailed proof is presented in Appendix \ref{prf-convergence-cg}.\vspace{-5pt}

\begin{cor}\label{thm-converge-cg}
CG will converge to the optimum of \textbf{WCEP} in a finite number of iterations, which is bounded by $3^{m_{\xi}}\cdot S$.\vspace{-5pt}
\end{cor}

\begin{remark}
	For $s_1\ne s_2$, \textbf{PSP}$_{s_1}$ and \textbf{PSP}$_{s_2}$ are independent. This provides an opportunity for us to perform parallel computation for \textbf{PSP}$_{s}$'s to accelerate the CG procedure.
\end{remark}

\section{Solving DRD-MEMG via C\&CG-DRO(CG)}\label{ccg}
With the aforementioned CG procedure to solve \textbf{WCEP}, we next integrate it within the classical C\&CG framework \cite{zeng2013solving} to develop a complete C\&CG-DRO(CG) algorithm for the complex \textbf{DRD-MEMG}, which, according to \cite{lu2024}, has a significantly strong computational capacity. Let $\Upsilon_s$ represent the set of scenarios generated up to the current iteration for $s\in [S]$. Then, the main master problem, denoted as \textbf{MMP}, is formulated as follows. 
\setlength{\arraycolsep}{-0.5em}
\begin{eqnarray}
	&&\textbf{MMP}:~\underline{w} =\; \min_{\mathbf{x},\eta,\eta_{\bm{\xi}_s},\mathbf{y}_{\bm{\xi}_s}} ~\mathbf{c}^{\intercal}\mathbf{x}+\eta \label{mmp-1} \\ 
	&&\mathrm{s.t.} ~ 
	\mathbf{A}\mathbf{x}\leq\mathbf{b} \label{mmp-2}  \\
	&&\quad\;\;\;\eta\geq\max_{\pi_{\bm{\xi}_s}\in\mathbb{R}_+}~{\textstyle\sum_{s=1}^S}\Big\{\textstyle\pi_s^{\rm e}\sum_{\bm{\xi}_s\in\Upsilon_s}\eta_{\bm{\xi}_s}\pi_{\bm{\xi}_s}:\label{mmp-3}\\
	&&\quad\;\;\;\textstyle\sum_{\bm{\xi}_s\in\Upsilon_s}\pi_{\bm{\xi}_s}=1:\quad(\alpha_s\in\mathbb{R})\label{mmp-4}\\
	&&\quad\;\;\;\textstyle\sum_{s=1}^{S}\sum_{\bm{\xi}_s\in\Upsilon_s}\pi_s^{\rm e}\left\|\bm{\xi}_s-\bm{\xi}_s^{\rm e} \right\|_1\pi_{\bm{\xi}_s}\leq r:\;(\beta\in\mathbb{R}_+)\Big\}\label{mmp-5} \\
	&&\quad\;\;\;\eta_{\bm{\xi}_s}=\mathbf{d}^{\intercal}\mathbf{y}_{\bm{\xi}_s},\quad \forall\bm{\xi}_s\in\Upsilon_s,\forall s\in[S] \label{mmp-6}   \\
	&&\quad\;\;\; \mathbf{F}\mathbf{y}_{\bm{\xi}_s}\leq\mathbf{h}-\mathbf{G}\mathbf{x}-\mathbf{K}\bm{\xi}_s,\quad \forall\bm{\xi}_s\in\Upsilon_s,\forall s\in[S]\label{mmp-7}\\
	&&\quad\;\;\;\mathbf{x}\in\left\lbrace 0,1\right\rbrace ^{n_{\rm x}}\times\mathbb{R}^{m_{\rm x}}_{+},\;\mathbf{y}_{\bm{\xi}_s}\in\mathbb{R}_{+}^{m_{\rm y}} \label{mmp-8}
\end{eqnarray}
Note that \eqref{mmp-1}--\eqref{mmp-8} is a BLO problem. By the strong duality of its lower-level LP, an equivalent single-level counterpart of \textbf{MMP} is derived below, which is adopted in our algorithm.
\begin{eqnarray}
	&&\underline{w} =\; \min_{\mathbf{x},\eta,{\alpha}_s,{\beta},\eta_{\bm{\xi}_s},\mathbf{y}_{\bm{\xi}_s}} ~\mathbf{c}^{\intercal}\mathbf{x}+\eta \label{mmp-9} \\ 
	&&\mathrm{s.t.} ~ 
	 \eqref{mmp-2},\;\eqref{mmp-6}\text{--}\eqref{mmp-8}\label{mmp-10}  \\
	&&\quad\;\;\;  \eta\geq\textstyle\sum_{s=1}^S{\alpha}_s+r{\beta},\;\;\alpha_s\in\mathbb{R}, \beta\in\mathbb{R}_+ \label{mmp-11}  \\
	&&\quad\;\;\;{\alpha}_s+\pi_s^{\rm e}\left\|\bm{\xi}_s-\bm{\xi}_s^{\rm e} \right\|_1{\beta}\geq\pi_s^{\rm e}\eta_{\bm{\xi}_s},\;\; \forall \bm{\xi}_s\in\Upsilon_s,\forall s\in[S]\label{mmp-12}
\end{eqnarray}
Obviously, \textbf{MMP} is a relaxation of \textbf{DRD-MEMG}, resulting in a  lower bound (LB) to $w$, i.e., 
\begin{equation}
	\textrm{LB}=\underline{w}\leq w.\label{def-lb}
\end{equation}

With an optimal $\hat{\mathbf{x}}$ output from \textbf{MMP}, we solve \textbf{WCEP} by using the CG procedure in Section \ref{wcep-cg} to generate a set of new \textit{optimality cuts}. Specifically, for $s\in[S]$, let $\widehat{\Omega}_s(\hat{\mathbf{x}})$ denote the optimal scenario set obtained in CG. For each $s$, we update $\Upsilon_s\leftarrow\Upsilon_s\cup\widehat{\Omega}_s(\hat{\mathbf{x}})$, and introduce variables $\{\mathbf{y}_{\hat{\bm{\xi}}_{s}},\eta_{\hat{\bm{\xi}}_{s}}\}_{\bm{\xi}_s\in\widehat{\Omega}_s(\hat{\mathbf{x}})}$ and constraints \eqref{new-cut-1}--\eqref{new-cut-3} to strengthen \textbf{MMP}:
\begin{align}
	&{\alpha}_s+\pi_s^{\rm e}\|\hat{\bm{\xi}}_{s}-\bm{\xi}_s^{\rm e} \|_1{\beta}\geq\pi_s^{\rm e}\eta_{\hat{\bm{\xi}}_{s}},\quad \forall \bm{\xi}_s\in\widehat{\Omega}_s(\hat{\mathbf{x}}) \label{new-cut-1} \\
	&\eta_{\hat{\bm{\xi}}_{s}}=\mathbf{d}^{\intercal}\mathbf{y}_{\hat{\bm{\xi}}_{s}},\;\mathbf{y}_{\hat{\bm{\xi}}_{s}}\in\mathbb{R}_{+}^{m_{\rm y}},\quad \forall \bm{\xi}_s\in\widehat{\Omega}_s(\hat{\mathbf{x}}) \label{new-cut-2}  \\
	& \mathbf{F}\mathbf{y}_{\hat{\bm{\xi}}_{s}}\leq\mathbf{h}_{\hat{\bm{\xi}}_{s}}-\mathbf{G}\mathbf{x}-\mathbf{K}\hat{\bm{\xi}}_{s},\quad \forall \bm{\xi}_s\in\widehat{\Omega}_s(\hat{\mathbf{x}}).\label{new-cut-3}
\end{align}
Meanwhile, an upper bound (UB) to $w$ can be updated as:
\begin{align}
\mathrm{UB}=\min\{\mathrm{UB}\;,\;\mathbf{c}^{\intercal}\hat{\mathbf{x}}+v(\hat{\mathbf{x}})\}\geq w.\label{def-ub}
\end{align}

The entire C\&CG-DRO(CG) iteratively solves \textbf{MMP} and \textbf{WCEP} until the relative solution gap ($\frac{\mathrm{UB}-\mathrm{LB}}{\left| \mathrm{LB}\right|}\times100\%$) falls below a pre-defined tolerance $\varepsilon\geq 0$. 

The overall flow of our algorithm is described in Appendix \ref{alg-overflow}, which exhibits a nested architecture. In the outer loop, we gradually augment a relaxation through the C\&CG framework to approach the optimum of \textbf{DRD-MEMG}. The inner one adopts CG as the subroutine to derive  \textbf{WCEP}'s optimal solution. Hence, the complete algorithm is referred to as C\&CG-DRO(CG), whose convergence result is established in the following. Please see the detailed proof in Appendix \ref{prf-convergence-ccg}.\vspace{-5pt}

\begin{cor}\label{thm-converge-ccg}
	C\&CG-DRO(CG) will converge to the global optimum of \textbf{DRD-MEMG} within $3^{m_{\xi}}\cdot S$ iterations.
	\vspace{-5pt}
\end{cor}

\begin{remark}\label{remark-strategy}
	Based on the specific structure of \textbf{DRD-MEMG}, we customize the following enhancement strategies in C\&CG-DRO(CG) implementation. Yet, in theoretical analysis, we still adopt the original algorithm version.
	
	i) \textbf{Warm Start}: The initialization of ${\Upsilon}_s$'s and ${\Omega}_s(\hat{\mathbf{x}})$'s is important for C\&CG-DRO(CG). On the one hand, they must guarantee the feasibility of \eqref{mmp-12} and \eqref{pmp-3}. On the other hand, they may affect the convergence efficiency. Hence, we initialize ${\Upsilon}_s=\{\bm{\xi}^{\rm e}_s\}$, and, in  CG, $\Omega_s(\hat{\mathbf{x}})$ directly inherits the current ${\Upsilon}_s$ before we start CG. 
	
	ii) \textbf{Solution Space Reduction for \textbf{PSP}$_s$}: Let $\chi_{\rm wt}^t/\chi_{\rm pv}^t$ be the multipliers of \eqref{cons-2stg-5}. Involving them in \textbf{PSP}$_s$'s objective function by dualizing the recourse problem, we have\vspace{-5pt}
	$$
	\max_{\tilde{\phi}_k^t,\chi_k^t}~\textstyle\sum_{t\in[t]}\sum_{k\in\{{\rm wt,pv}\}}(-\tilde{\phi}^t_{k}P_{k}\cdot \chi_{k}^t-\beta |\tilde{\phi}^t_{k}-\tilde{\phi}^{{\rm e},t}_{k} |).\vspace{-5pt}
	$$
	Given that $\chi_{k}^t,\beta\geq0$, and considering the maximization direction, the optimal solution of $\tilde{\phi}_{\rm wt}^t/\tilde{\phi}_{\rm pv}^t$ is either its minimum value or just the sample's value, i.e., the optimal $\hat{\xi}_s^i$ associated with $\tilde{\phi}_{\rm wt}^t/\tilde{\phi}_{\rm pv}^t$ can be further restricted to $\{\underline{\xi}_s^i,{\xi}_s^{{\rm e},i}\}$.
	
	iii) \textbf{Scenario Selection}: In the execution of  CG, we notice that the number of scenarios, where the optimal distribution to \textbf{WCEP} is supported, is not more than $S+1$. This observation is consistent with Remark \ref{sparsity}-ii. To avoid redundant scenarios in \textbf{MMP} in the implementation of C\&CG-DRO(CG), we update $\Upsilon_s$'s by only including the scenarios with non-zero probabilities in $\widehat{\Omega}_s(\hat{\mathbf{x}})$'s.
\end{remark}

\section{Numerical Studies}\label{case}
Numerical studies have been conducted on an MEMG in an industrial park.  Main parameters for MEMG's dispatch are summarized in Appendix \ref{case-set}. The time interval ($\Delta_t$) is chosen as 0.5 hour (h), i.e., $T=48$. The forecasted values and the corresponding sample spaces of the random factors are also provided in Appendix \ref{case-set}. We list the electricity transaction prices in Table \ref{tab:e-price}, and the hydrogen procurement price is 5.724 \$/kg. As to the ambiguity set, we have simulated $S=10$ samples, each with $\pi_s=0.1$, and Wasserstein radius $r$ is set to 0.8. The numerical experiments have been implemented on a laptop with Intel i9-13900H processor (14 cores) and 16GB of RAM. C\&CG-DRO(CG) has been implemented by MATLAB R2023b with Gurobi 11.0. $M$ in \textbf{PSP}$_s$ is chosen as $10^{4}$, the termination tolerance $\varepsilon$ is selected as 0.5\%, and the time budget is set to 7200 seconds (s).

\begin{table}[!t]
  \centering
  \caption{Electricity Transaction Prices}
  \vspace{-5pt}
  	\resizebox{0.9\linewidth}{!}{%
  	\renewcommand{\arraystretch}{0.5}
    \begin{tabular}{ccccccccc}
    \toprule
    Time (h) & 1--8   & 9--10  & 11--14 & 15--17 & 18    & 19--22 & 23    & 24 \\
    \midrule
    $c_{\rm e,buy}^t$ (\$/kWh) & 0.0431  & 0.1135  & 0.1875  & 0.1135  & 0.1875  & 0.2058  & 0.1875  & 0.1140  \\
    $c_{\rm e,sell}^t$ (\$/kWh) & 0.0345  & 0.0908  & 0.1500  & 0.0908  & 0.1500  & 0.1646  & 0.1500  & 0.0912  \\
    \bottomrule
    \end{tabular}%
}\vspace{-15pt}
  \label{tab:e-price}%
\end{table}%

\subsection{MEMG's Dispatch Results}\label{result}

\subsubsection{Dispatch Scheme} 
By using our DRD approach, the day-ahead dispatch scheme of MEMG is visualized in Fig. \ref{pre-result}. As an RE-enriched system, solar and wind energy constitute MEMG's primary sources, generating approximately 58,131.41 kWh to fulfill multiple energy demands. They have complementary characteristics. As shown in Fig. \ref{sub-e}, WT generation is higher during the night-time and early-morning periods, specifically from 1 to 7 AM and from 7 PM to midnight, and yet it is relatively lower during daytime slots. PV outputs, consistent with the trend of solar irradiation, remain zero during night-time periods, begin to increase from 7 AM, peak around noon, and then gradually decrease. Electricity transactions with the utility grid serve as a supplementary source of energy and provide opportunities for profit. In Fig. \ref{sub-e}, MEMG purchases electricity from 2 to 9 AM, from 3 to 6 PM, and at midnight. The electricity selling occurs during 11 AM--3 PM and 7 PM--12 AM at higher prices (i.e., 0.15 and 0.1646 \$/kWh as listed in Table \ref{tab:e-price}). ELZ, which is a P2HH device, operates to produce hydrogen, and coordinates with FC and HWT for heat supply, as illustrated in Fig. \ref{sub-t}. Notably, ELZ is mainly driven by surplus RE and supplimented by purchased electricity during off-peak periods to reduce cost. Moreover, the energy storage systems help compensate for the temporal mismatch between the provision and consumption, and operate for potential price arbitrage. For example, BSS discharges when electricity consumption or selling prices are high, e.g., from 9 AM to 3 PM and from 7 PM to 12 AM, while charges during the other time slots. The energy storage levels of BSS, HWT, and HT are shown in Fig. \ref{sub-s}.

\begin{figure}[!t]
	\centering
	\subfloat[Operation of electrical subsystem.]{
		\label{sub-e}
		\includegraphics[width=0.4\textwidth]{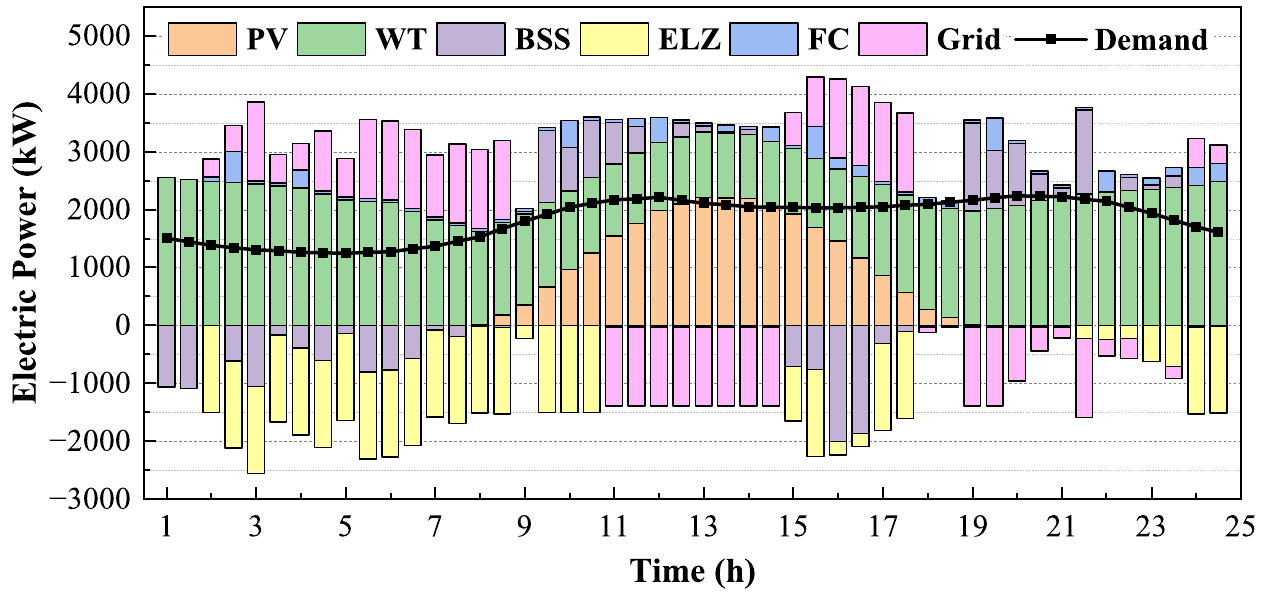}
	}\vspace{-10pt}\\
	\subfloat[Operation of thermal subsystem.]{
		\label{sub-t}
		\includegraphics[width=0.4\textwidth]{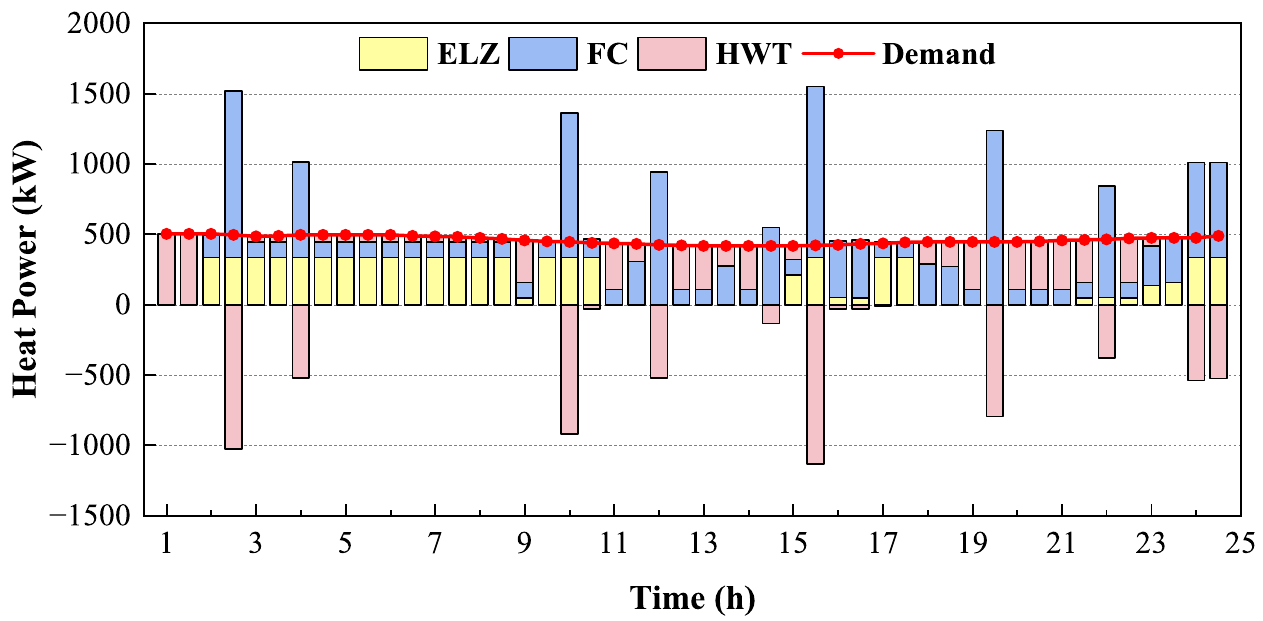}
	}\vspace{-10pt}\\
	\subfloat[Storage level of multiple energy systems.]{
		\label{sub-s}
		\includegraphics[width=0.4\textwidth]{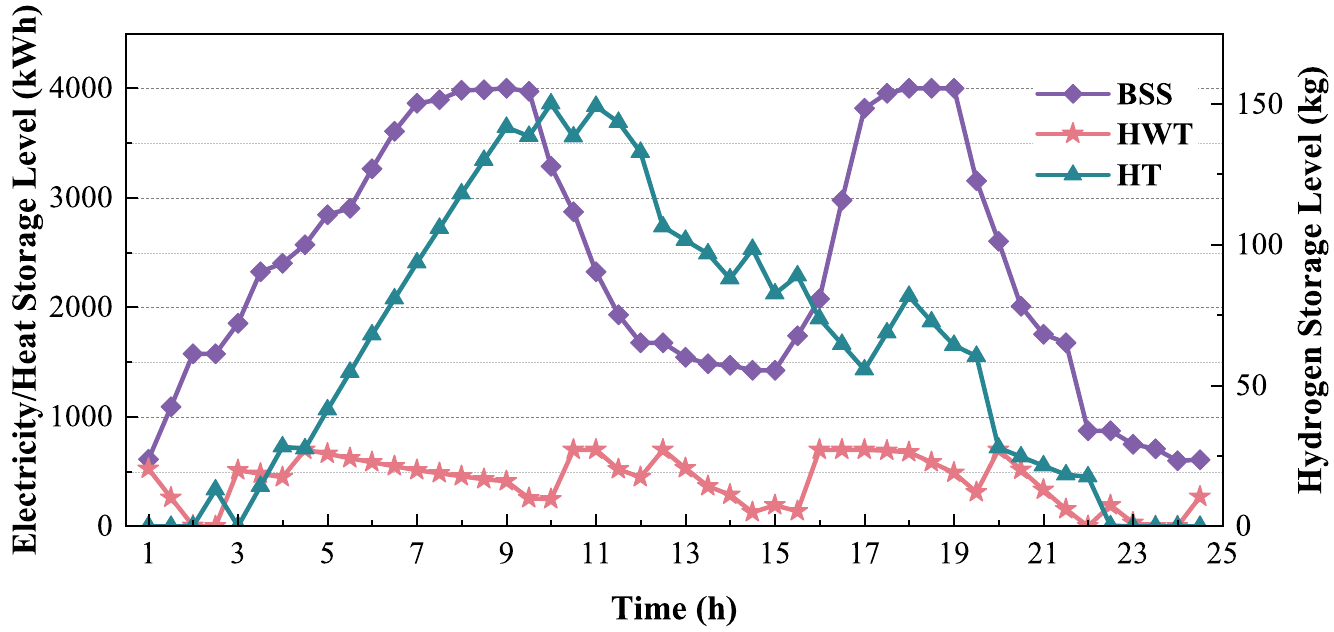}
	}
	\caption{MEMG's dispatch scheme.}
	\label{pre-result}
	\vspace{-20pt}
\end{figure}

\subsubsection{Out-of-Sample Performance} We have evaluated the \textit{out-of-sample performance} \cite{mohajerin2018data} of MEMG's dispatch scheme on another test set that comprises 1000 scenarios. The solutions by employing SP and RO have also been presented. We should note that in our numerical studies, the SP scheme specifically refers to the one derived through sample-average approaximation (SAA) \cite{stochastic1} with the given empirical distribution $\mathbb{P}^{\rm e}$, and, for RO, we directly take the sample space $\Xi$ being the uncertainty set, based on which the optimal worst-case solution is attained.

Table \ref{tab:oos} displays the results. We have selected six indices for analysis, including the out-of-sample cost (OOSC), probabilities of electricity and heat load shedding (PELS and PHLS), expected electrical and heat energy not supplied (EEENS and EHENS), and expected net CO\textsubscript{2} emissions (ENCE). Their definitions can be found in Appendix \ref{index-oos}. In Table \ref{tab:oos}, the dispatch schemes obtained by DRO and SP have negative OOSC values, indicating that they are profitable. Compared to SP, since DRO accounts for the ambiguity of the underlying probability distribution $\mathbb{P}$ and hedges against it in a robust way, the OOSC of DRO is \$1.68 higher, while the PELS and PHLS decrease, respectively, from 8.7\% to 0.3\% and from 12.1\% to 1.2\%. Besides, DRO's EEENS and EHENS are 0.08 kWh and 0.54 kWh, representing declines of 99.33\% and 95\% than those of SP. They demostrate DRO's advantages in uncertainty modeling. As to RO, although it meets all the energy demand, the OOSC incurred, i.e., \$440.19, is substantial. The over-conservativeness makes RO not recommended for practical modeling of MEMG's dispatch. On the other hand, DRO's scheme possesses a lower ENCE (related to electricity transactions) than those of SP and RO, suggesting DRO's ability to manage uncertain risks also contributes to the decarbonization of the industrial park.

\begin{table}[!t]
  \centering
  \caption{Out-of-Sample Evaluation}
  \vspace{-5pt}
  \resizebox{0.8\linewidth}{!}{%
  	\renewcommand{\arraystretch}{1.2}
    \begin{tabular}{ccccccc}
    \toprule
    Approach & OOSC (\$) & PELS  & PHLS  & EEENS (kWh)  & EHENS (kWh) & ENCE (kg) \\
    \midrule
    DRO   & -590.97 & 0.3\% & 1.2\% & 0.08  & 0.54  & 862.48 \\
    SP    & -592.65 & 8.7\% & 12.1\% & 11.87  & 10.80   & 867.21 \\
    RO    & 440.19 & 0     & 0     & 0     & 0     & 4539.14 \\
    \bottomrule
    \end{tabular}%
}\vspace{-10pt}
  \label{tab:oos}%
\end{table}%

\subsubsection{Analysis on ELZ's model}
\begin{figure}[]
	\centering
	\includegraphics[width=0.4\textwidth]{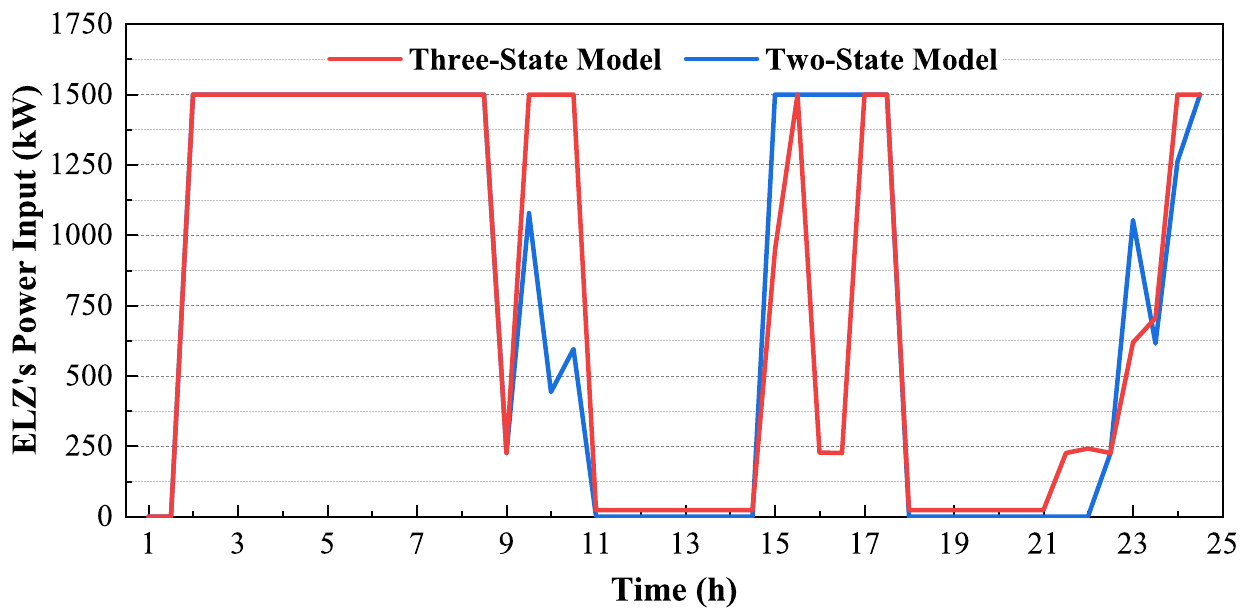}
	\vspace{-10pt}
	\caption{ELZ's dispatch schemes by using three- and two-state models.}
	\label{elz-result}
	\vspace{-15pt}
\end{figure}

Fig. \ref{elz-result} presents the dispatch results of ELZ by using our three-state model (3SM) and the traditional two-state model (2SM). It can be observed that, from 11 AM to 2:30 PM and from 6 to 9 PM, ELZ is in the standby state when using 3SM. If we adopt 2SM, ELZ remains idle during the same periods yet has to wait until 10 PM to operate due to the cold startup. In contrast to 2SM, 3SM reduces ELZ's startup and shutdown costs by \$39.20 and lowers OOSC by \$12.41, reflecting the value of the operating flexibility it provides for ELZ.

We have also compared our P2HH model with the traditional power-to-hydrogen (P2H) one for ELZ, in which the waste heat recovery is not taken into account. By conducting out-of-sample evaluation for MEMG's dispatch scheme with P2H model, it appears an OOSC of \$215.86, a PHLS of 99.1\%, an EHENS of 545.95 kWh, and an ENCE of 3,104.30 kg, which are much worse than those by using P2HH model in Table \ref{tab:oos}. Hence, they demonstrate the necessity to collect ELZ's waste heat to improve MEMG's operating efficiency.\vspace{-5pt}

\subsection{Impact of Wasserstein Radius}

Fig. \ref{r-result} depicts the variation of \textbf{DRD-MEMG}'s optimal objective value $w$ (a.k.a, the in-sample cost) and OOSC with respect to the radius $r$ of the Wasserstein ball $\mathcal{P}$. The optimal objective values ($w_{\rm sp}$ and $w_{\rm ro}$) obtained by SP and RO are also presented for reference. As can be visualized in the figure, $w$ typically increases with $r$, which is consistent with Proposition \ref{prop-r}, and is bounded by $w_{\rm sp}$ (-\$611.69) and $w_{\rm ro}$ (\$815.45), demonstrating that DRO is a trade-off between SP and RO.  Besides, $r$ manifests the risk preference of the MEMG's operator. A negative OOSC (revenue) is derived if $r$ is smaller than around 5.4; otherwise, OOSC is positive (expense). An optimistic MEMG operator may be inclined to select a small $r$, whereas a pessimistic one may prefer a relatively large $r$ to attain a risk-averse dispatch. Note that there is a huge range for OOSC from -\$592.78 through \$440.19, almost growing by \$1,032.97. Hence, Wasserstein radius $r$ serves as an important parameter for the economic benefit of MEMG's dispatch. Moreover, unlike the continuous growth of $w$, it is interesting to observe a deferred worsening (increase) in OOSC. \vspace{-5pt}

\begin{figure}[]
	\centering
	\includegraphics[width=0.4\textwidth]{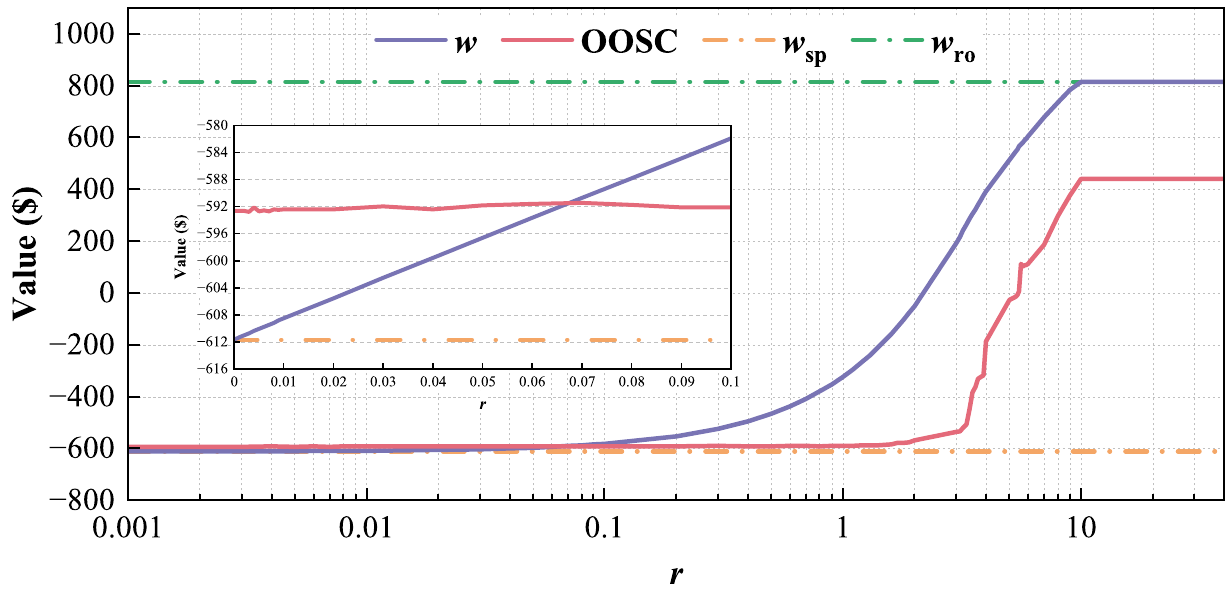}
	\vspace{-10pt}
	\caption{Impact of Wasserstein radius.}
	\label{r-result}
	\vspace{-15pt}
\end{figure}

\begin{remark}
	According to Proposition \ref{prop-psp}, each component of a generated scenario $\bm{\xi}_s$  has only three candidate values, i.e., $\overline{\xi}^{i}$, $\underline{\xi}^{i}$, and $\xi_{s}^{{\rm e},i}$. On one hand, when $r=0$, i.e., the ambiguity-free case, we have $\xi_{s}^{i}=\xi_{s}^{{\rm e},i}$, resulting in the solution to SAA-based SP. On the other hand, when $r$ exceeds a critical threshold (e.g., $r\geq10$ in Fig. \ref{r-result}),  \eqref{mmp-5} in \textbf{MMP} and \eqref{pmp-3} in \textbf{PMP} can always be satisfied, rendering their dual multipliers $\beta=0$. In such a case, the generated scenarios from \textbf{PSP}$_s$ are further restricted in $\times_{i=1}^{m_{\xi}}\{\overline{\xi}^{i},\underline{\xi}^{i}\}$. Also,  C\&CG-DRO(CG) reduces to the classical C\&CG for two-stage RO problems  \cite{zeng2013solving}. Hence, the derived solution is equivalent to that of RO.\vspace{-5pt}
\end{remark}

\subsection{Computational Performance}

Computational tests on C\&CG-DRO(CG) to solve \textbf{DRD-MEMG} with different values of $r$ and $S$ have been performed and analyzed. Two other typical algorithms for DRO, i.e., the primary C\&CG implementation to solve the duality-based DRO reformulation \cite{gamboa2021decomposition,saif2021data} (referred to as basic C\&CG) and the Benders-dual algorithm  \cite{gamboa2021decomposition,duque2022distributionally} (referred to as Benders), have been selected for comparison. Numerical results are given in Table. \ref{tab:com-result}. Columns ``UB'', ``LB'', ``Gap'', ``Iter'', ``$|\Upsilon|$'', and ``Time'' report, respectively, the UB, LB, relative solution gap, number of (outer) iterations, total number of generated scenarios ($|\Upsilon|\triangleq\sum_{s=1}^S|\Upsilon_s|$),  and solution time, all upon termination. For Benders, Column ``Cut'', representing the number of cutting planes, replaces its counterpart Column ``$|\Upsilon|$''. We note that parallelization has been implemented across all these algorithms. For ease of exposition, we denote the instance with radius $r$ and $S$ samples by $\mathcal{I}_r^S$. 

Based on the results in Table \ref{tab:com-result}, we can make a few interesting observations and develop a set of critical insights regarding these algorithms: 
\begin{itemize}
	\item C\&CG-DRO(CG) showcases a strong and scalable  solution capacity for \textbf{DRD-MEMG}. It handles all test instances within 4,000 seconds. Note that C\&CG-DRO(CG) derives optimal solutions with 2 or 3  iterations, indicating that the generated scenarios through inner CG are highly effective in capturing worst-case distributions for non-trivial first stage decisions. A general trend is that the solution time increases both with $r$ and $S$. We believe this is quite intuitive, noting that a larger $r$ indicates a larger solution space and a larger $S$ leads to larger \textbf{MMP}/\textbf{PMP} and more \textbf{PSP}$_s$'s.

	\item Between C\&CG-DRO(CG) and basic C\&CG, the former one clearly outperforms the latter one. For the very small-scale instances with just 3 or 5 samples, they are basically comparable. Nevertheless, for large-scale ones, e.g., $\mathcal{I}_{0.5}^{400}$ and $\mathcal{I}_{1}^{400}$, C\&CG-DRO(CG) nearly achieves a speedup of one order of magnitude. In fact, basic C\&CG can be seen as a special case of C\&CG-DRO(CG) if we simply perform one CG iteration for every \textbf{MMP} \cite{lu2024}. Hence, the deep search achieved by executing a complete CG procedure provides an accurate evaluation of \textbf{WCEP} and contributes to generating all critical scenarios to strengthen \textbf{MMP}. Actually, in each CG execution, we often observe that the number of scenarios of non-zero probabilities is often $S+1$, reflecting the result in Remark \ref{sparsity}-ii. Also, the scenario selection strategy in Remark \ref{remark-strategy}-iii helps accelerate the solution procedure through reducing the redundancy of \textbf{MMP}. Moreover, C\&CG-DRO(CG) demonstrates a strong numerical stability, noting that the gap between UB and LB generally reduces to zero regardless of the $\varepsilon=0.5\%$ optimality tolerance.

	\item Benders is completely dominated by C\&CG-DRO(CG) and  basic C\&CG, displaying a very weak computational capacity. When it derives an optimal solution, the overall solution time could be 2--3 orders of magnitude longer than those of C\&CG-DRO(CG) and  basic C\&CG. Actually,  when $r=5$, Benders fails to generate a high-quality feasible solution within the time limit for all instances.
\end{itemize}

\begin{table*}[htbp]
	\centering
	\caption{Results of Computational Tests}
	\vspace{-5pt}
	\resizebox{0.88\linewidth}{!}{%
		\begin{threeparttable}
			\renewcommand{\arraystretch}{0.9}
			\begin{tabular}{ccccccccccccccccccccccc}
				\toprule
				\multirow{1.4}[4]{*}{$r$} & \multirow{1.4}[4]{*}{$S$} & \multirow{21}[10]{*}{} & \multicolumn{6}{c}{C\&CG-DRO(CG)}               & \multirow{21}[10]{*}{} & \multicolumn{6}{c}{Basic C\&CG}                 & \multirow{21}[10]{*}{} & \multicolumn{6}{c}{Benders} \\
				\cmidrule{4-9}\cmidrule{11-16}\cmidrule{18-23}          &       &       & UB    & LB    & Gap   & Iter  & $|\Upsilon|$   & Time (s) &       & UB    & LB    & Gap   & Iter  & $|\Upsilon|$   & Time (s) &       & UB    & LB    & Gap   & Iter  & Cut   & Time (s) \\
				\cmidrule{1-2}\cmidrule{4-9}\cmidrule{11-16}\cmidrule{18-23}    \multirow{6.4}[2]{*}{0.5} & 3     &       & -556.01  & -556.33  & 0.06\% & 2     & 6     & 8.45  &       & -555.81  & -558.20  & 0.43\% & 4     & 12    & 9.11  &       & -554.86  & -557.62  & 0.49\% & 228   & 684   & 1651.45  \\
				& 5     &       & -467.66  & -467.66  & 0     & 2     & 9     & 10.40  &       & -467.27  & -468.62  & 0.29\% & 4     & 20    & 13.36  &       & -466.63  & -468.44  & 0.39\% & 176   & 880   & 2001.39  \\
				& 20    &       & -482.15  & -482.15  & 0     & 2     & 26    & 37.75  &       & -481.48  & -483.83  & 0.48\% & 4     & 80    & 82.52  &       & -480.95  & -483.27  & 0.48\% & 70    & 1400  & 1284.46  \\
				& 50    &       & -480.56  & -480.76  & 0.04\% & 2     & 68    & 178.47  &       & -480.53  & -480.83  & 0.06\% & 5     & 250   & 676.27  &       & -479.36  & -481.54  & 0.45\% & 42    & 2100  & 1260.05  \\
				& 100   &       & -474.32  & -474.32 & 0     & 2     & 109   & 387.27  &       & -474.24  & -474.31 & 0.01\% & 5     & 500   & 1362.67  &       & -473.12  & -475.32 & 0.46\% & 34    & 3400  & 2156.95  \\
				& 200   &       & -481.07  & -481.07  & 0     & 2     & 237   & 583.63  &       & -480.99  & -481.08  & 0.02\% & 5     & 1000  & 3839.14  &       & -480.06  & -481.99  & 0.40\% & 33    & 6600  & 4425.54  \\
				& 400   &       & -482.72  & -482.72  & 0     & 2     & 436   & 1407.52  &       & -455.38  & -495.86  & 8.16\% & 3     & 1200  & T (12906.26) &       & -481.10  & -484.98  & 0.80\% & 25    & 10000 & T \\
				\cmidrule{1-2}\cmidrule{4-9}\cmidrule{11-16}\cmidrule{18-23}    \multirow{6.4}[2]{*}{1} & 3     &       & -414.84  & -414.84  & 0     & 3     & 7     & 11.05  &       & -413.84  & -415.06  & 0.29\% & 6     & 18    & 14.34  &       & -413.51  & -415.54  & 0.49\% & 229   & 687   & 1753.28  \\
				& 5     &       & -322.64  & -322.64  & 0     & 2     & 9     & 7.86  &       & -322.13  & -322.43  & 0.09\% & 5     & 25    & 17.92  &       & -321.33  & -322.94  & 0.50\% & 187   & 935   & 2269.19  \\
				& 20    &       & -340.28  & -340.28  & 0     & 2     & 31    & 38.43  &       & -339.66  & -339.81  & 0.05\% & 5     & 100   & 173.20  &       & -339.00      & -340.25      & 0.37\%      & 83      & 1660      & 2022.99 \\
				& 50    &       & -339.13  & -339.13  & 0     & 3     & 97    & 433.72  &       & -338.86  & -339.09  & 0.07\% & 5     & 250   & 723.05  &       & -338.37      & -339.63      & 0.37\%      & 52      & 2600      & 2489.97 \\
				& 100   &       & -332.10  & -332.47  & 0.11\% & 2     & 150   & 547.21  &       & -331.93  & -332.06  & 0.04\% & 5     & 500   & 1585.28  &       & -331.57      & -332.59      & 0.31\%      & 44      & 4400      & 4955.64 \\
				& 200   &       & -338.32  & -339.66  & 0.40\% & 2     & 300   & 883.36  &       & -338.91  & -338.99  & 0.02\% & 5     & 1000  & 5363.19  &       & -334.06      & -342.90      & 2.58\%      & 28      & 5600      & T \\
				& 400   &       & -340.18  & -341.20  & 0.30\% & 2     & 600   & 2420.65  &       & -308.04  & -339.41  & 15.62\% & 3     & 1200   & T (17357.74) &       & -320.62      & -360.11      & 10.97\%      & 13      & 5200      & T \\
				\cmidrule{1-2}\cmidrule{4-9}\cmidrule{11-16}\cmidrule{18-23}    \multirow{4.5}[2]{*}{5} & 3     &       & 446.00  & 446.00  & 0     & 3     & 7     & 24.09  &       & 446.65  & 446.00  & 0.15\% & 4     & 12    & 29.43  &       & 467.04  & 430.29  & 8.54\% & 262   & 786   & T \\
				& 5     &       & 498.86  & 498.86  & 0     & 3     & 11    & 39.58  &       & 499.94  & 498.86  & 0.22\% & 4     & 20    & 43.52  &       & 519.65  & 478.66  & 8.56\% & 193   & 965   & T \\
				& 20    &       & 483.94  & 483.94  & 0     & 3     & 37    & 245.34  &       & 484.90  & 483.43  & 0.31\% & 4     & 80    & 601.95  &       & 513.22  & 461.36  & 11.24\% & 110   & 2200  & T \\
				& 50    &       & 479.09  & 479.09  & 0     & 3     & 102   & 1011.54  &       & 480.33  & 479.35  & 0.20\% & 4     & 200   & 1890.29  &       & 526.23  & 442.49  & 18.92\% & 73    & 3650  & T \\
				& 100   &       & 483.43  & 483.43  & 0     & 3     & 185   & 3958.62  &       & 517.46  & 450.83  & 14.78\% & 3     & 300   & T (7726.95) &       & 551.26  & 423.37  & 30.21\% & 40    & 4000  & T \\
				\bottomrule
			\end{tabular}%
			\begin{tablenotes}
				\item ``T" implies the test instance terminates due to the time limit, i.e., 7200 s. In such situation, for basic C\&CG, we further provide the actual solution time in brackets.
			\end{tablenotes}
		\end{threeparttable}
	}\vspace{-15pt}
	\label{tab:com-result}%
\end{table*}%

Finally, we have examined the effectiveness of parallel computing for CG. Let $\mathfrak{t}(q)$ be the total solution time of C\&CG-DRO(CG) with $q$ threads being used for solving \textbf{PSP}$_s$'s in parallel. Fig. \ref{fig:parallel} depicts the speedup  ($\frac{\mathfrak{t}(1)}{\mathfrak{t}(q)}$) with respect to the number of threads. It is clear that the speedup increases along with a larger number of threads, demonstrating the nontrivial advantage of parallel computation. Typically, more threads yield diminishing returns, and, according to Amdahl's law, there exists theoretically  a limit for speedup that is related to the unparallelizable portion of C\&CG-DRO(CG). Additionally, the speedup decreases with Wasserstein radius, noting that \textbf{MMP}'s, an unparallelizable portion, are  more challenging to compute for larger $r$'s.   \vspace{-15pt}

\begin{figure}[]
	\centering
	\includegraphics[width=0.4\textwidth]{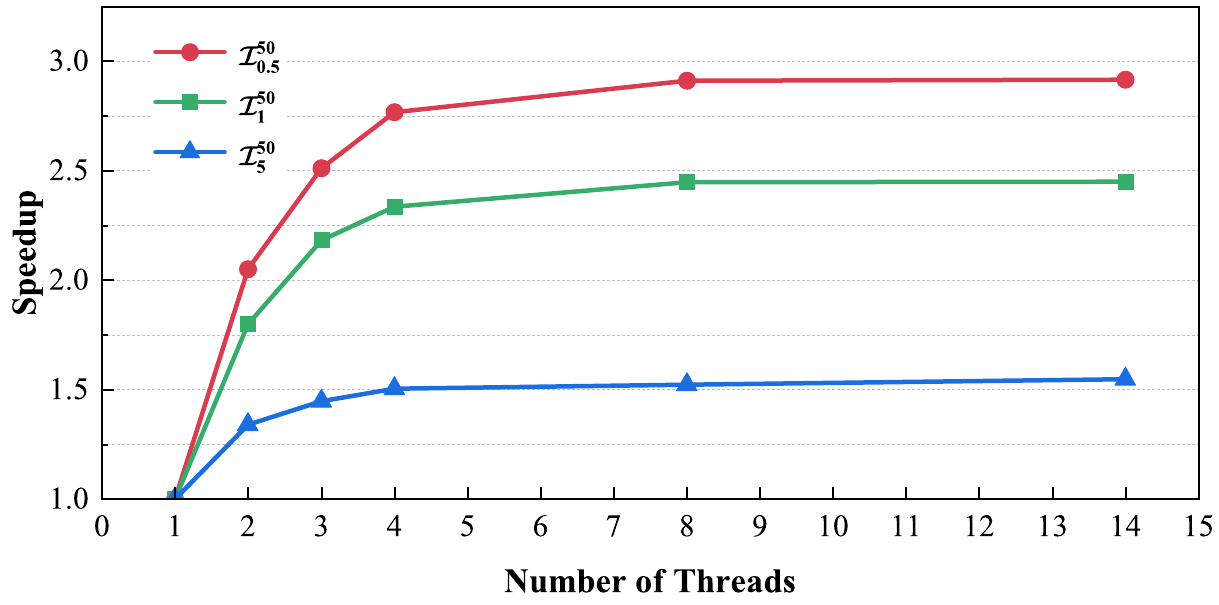}
	\vspace{-5pt}
	\caption{Parallelization speedup of C\&CG-DRO(CG) for $\mathcal{I}_{0.5}^{50}$, $\mathcal{I}_{1}^{50}$, and $\mathcal{I}_{5}^{50}$.}
	\label{fig:parallel}
	\vspace{-15pt}
\end{figure}

\section{Conclusion}\label{conclusion}
This paper has proposed a data-driven adaptive DRD formulation for MEMG, considering supply and demand uncertainties. A Wasserstein ambiguity set has been constructed to capture the unknown distributions around the empirical distribution, which is supported by available data. To address the computational burden, we have customized and developed a C\&CG-DRO(CG) algorithm for exact and high-efficient solution. Numerical studies have demonstrated the effectiveness of our DRD approach and elucidated the interrelationship of it with the traditional SP- and RO-based dispatch approaches. Moreover, the superiority of C\&CG-DRO(CG) has been verified by comparing it to two popular algorithms in the literature for DRO, i.e., basic C\&CG and the Benders-dual method.

\bibliographystyle{IEEEtran}
\bibliography{BibDRO}

\clearpage

\appendices

\section{Proof of Corollary \ref{thm-converge-cg}}\label{prf-convergence-cg}
\begin{proof}
	Assume that, initially, $\Omega_s(\hat{\mathbf{x}})=\emptyset$ for all $s\in[S]$. First, we state the following claim. 
	
	\begin{cla}\label{cla-cg}
		In each iteration, CG either solves \textbf{WCEP} to an optimal solution or, for some $s'\in[S]$, generates a new scenario $\bm{\xi}_{s'}$  that does not belong to $\Omega_{s'}(\hat{\mathbf{x}})$.
	\end{cla}
	\begin{proof}[Proof of Claim \ref{cla-cg}]
		We note that solving $\textbf{PMP}$ equivalently solves \textbf{WCEP} once \textbf{PSP}$_{s}$ returns non-positive reduced cost for all $s\in [S]$. Hence, it is sufficient to consider the case where \textbf{PSP}$_{s'}$ generates a scenario $\bm{\xi}_{s'}'$ of a positive reduced cost for some $s'$, i.e., 
		\begin{align}
			\mu(\hat{\mathbf{x}})=\pi_s^{\rm e}\left(Q(\hat{\mathbf{x}},\bm{\xi}_{s'}'\right)-\hat{\beta}\left\|\bm{\xi}_{s'}'-\bm{\xi}_s^{\rm e}\right\|_1)-\hat{\alpha}_s>0.\label{zzz}
		\end{align}
		
		Since $\textbf{PMP}$ has been solved to  optimality, for all scenarios $\bm \xi_{s}\in \Omega_{s}(\hat{\mathbf{x}})$ ($\forall s\in[S]$), their corresponding reduced costs are less than or equal to zero. Hence, we can conclude that $\bm{\xi}_{s'}'$ is not in $\Omega_{s'}(\hat{\mathbf{x}})$. 
	\end{proof}

	 According to the proof of Proposition \ref{prop-psp}, the number of the extreme points of $\widetilde{\Xi}_s$ is finite, i.e.,  $3^{m_{\xi}}$. Also, based on the construction of $\textbf{PSP}_s$ presented in Corollary \ref{cor-PSP}, the optimal $(\hat{\bm{\xi}}_s,\hat{\mathbf{z}}_s)$ is an extreme point of $\widetilde{\Xi}_s$. Together with Claim \ref{cla-cg}, it is clear that the number of CG’s iterations is bounded by $3^{m_{\xi}}\cdot S$. That completes the proof.
\end{proof}

\section{Overall flow of C\&CG-DRO(CG)}\label{alg-overflow}
Fig. \ref{ccg-dro} describes the overall flow of C\&CG-DRO(CG). The dashed red box highlights the part that can be executed in parallel to accelerate the computation.

\begin{figure*}[]
	\centering
	\includegraphics[scale=0.38]{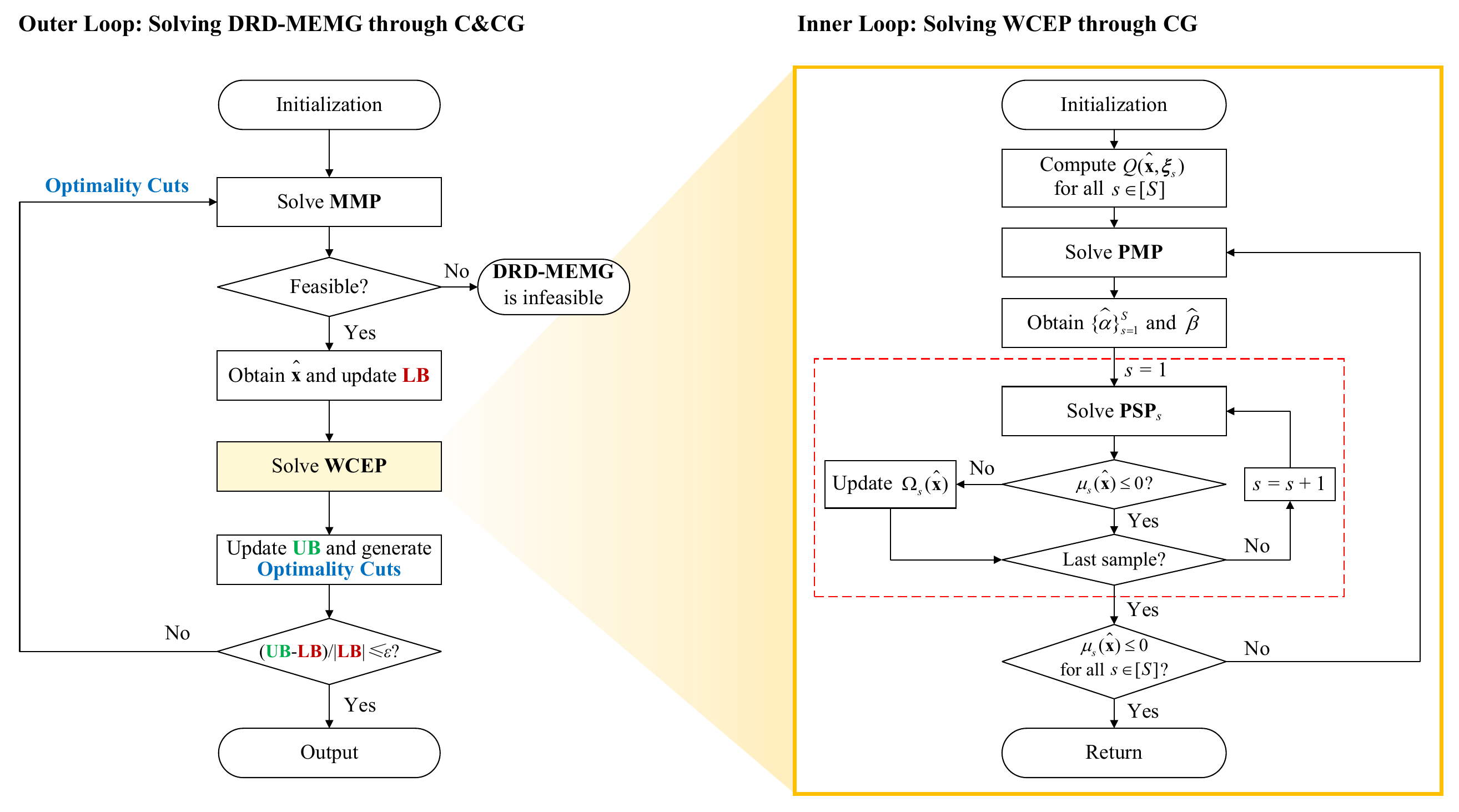}
	\caption{Overall flow of C\&CG-DRO(CG)}
	\label{ccg-dro}
\end{figure*}

\section{Proof of Corollary \ref{thm-converge-ccg}}\label{prf-convergence-ccg}
\begin{proof}
Assume that $\Upsilon_s$'s are set to $\emptyset$ at initialization. Consider the execution of C\&CG-DRO(CG) at some iteration, where $\Upsilon_s$'s ($\forall s\in[S]$) are the available scenario sets, $\hat{\mathbf{x}}$ is the optimal first-stage solution to \textbf{MMP},  and $\widehat{\Omega}_s(\hat{\mathbf{x}})$ is the set of  scenarios attained by solving \textbf{PSP}$_s$ to non-positive reduced cost. 
\begin{cla}\label{cla-ccg}
	If $\widehat{\Omega}_s(\hat{\mathbf{x}})\subseteq\Upsilon_s$ for all $s\in[S]$, we have $\mathrm{UB}=\mathrm{LB}$.
\end{cla}

\begin{proof}[Proof of Claim \ref{cla-ccg}]
	Consider the inner CG algorithm. For the obtained $\widehat{\Omega}_s(\hat{\mathbf{x}})$'s, 
	\begin{align}
		\hspace{-2mm}
		v(\hat{\mathbf{x}})=&\max_{\pi_{\bm{\xi}_s}\in\mathbb{R}_+}~\sum_{s=1}^S \sum_{\bm{\xi}_s\in\widehat{\Omega}_s(\hat{\mathbf{x}})}\pi_s^{\rm e}Q(\hat{\mathbf{x}},\bm{\xi}_{s})\pi_{\bm{\xi}_s}\label{pf-ccg-1}\\
		&~{\rm s.t.}~\sum_{\bm{\xi}_s\in\widehat{\Omega}_s}\pi_{\bm{\xi}_s}=1,\quad \forall s\in[S] \label{pf-ccg-2}\\
		&\quad\quad\;\;\sum_{s=1}^S\pi_s^{\rm e}\sum_{\bm{\xi}_s\in\widehat{\Omega}_s(\hat{\mathbf{x}})}\left\|\bm{\xi}_{s}-\bm{\xi}_s^{\rm e} \right\|_1\pi_{\bm{\xi}_s}\leq r\label{pf-ccg-2-1} \\
		=&\min_{\alpha_s\in\mathbb{R},\beta\in\mathbb{R}_+}~\sum_{s=1}^S\alpha_s+r\beta\label{pf-ccg-3}\\
		&~\nonumber{\rm s.t.}~\pi_s^{\rm e}(Q(\hat{\mathbf{x}},\bm{\xi}_s)-\beta\left\|\bm{\xi}_s-\bm{\xi}_s^{\rm e}\right\|_1)-\alpha_s\leq 0,\\
		&\qquad\qquad\qquad\qquad\quad\forall \bm{\xi}_s\in\widehat{\Omega}_s(\hat{\mathbf{x}}),\;\forall s\in[S]\label{pf-ccg-4}.
	\end{align}
	The second equality is obtained because of the strong duality of \textbf{PMP}. $\{{\alpha}_s\}_{s=1}^S$ and ${\beta\geq0}$ are multipliers of \eqref{pf-ccg-2} and \eqref{pf-ccg-2-1}, respectively.

	Shift perspective to the outer C\&CG algorithm,  we have
	\begin{align}
		\hspace{-4mm}
		\mathrm{LB}=&~\mathbf{c}^{\intercal}\hat{\mathbf{x}}+\min_{\alpha_s\in\mathbb{R},\beta\in\mathbb{R}_+}~\sum_{s=1}^S\alpha_s+r\beta\\
		&\nonumber\qquad\quad{\rm s.t.}~{\alpha}_s+\pi_s^{\rm e}\left\|\bm{\xi}_s-\bm{\xi}_s^{\rm e} \right\|_1 \beta\geq\pi_s^{\rm e}Q(\hat{\mathbf{x}},\bm{\xi}_s),\\
		&\qquad\qquad\qquad\qquad\qquad\forall \bm{\xi}_s\in\Upsilon_s,\forall s\in[S]\\
		\geq&~\mathbf{c}^{\intercal}\hat{\mathbf{x}}+\min_{\alpha_s\in\mathbb{R},\beta\in\mathbb{R}_+}~\sum_{s=1}^S\alpha_s+r\beta\\
		&\nonumber\qquad\quad{\rm s.t.}~{\alpha}_s+\pi_s^{\rm e}\left\|\bm{\xi}_s-\bm{\xi}_s^{\rm e} \right\|_1 \beta\geq\pi_s^{\rm e}Q(\hat{\mathbf{x}},\bm{\xi}_s),\\
		&\qquad\qquad\qquad\qquad\qquad\forall \bm{\xi}_s\in\widehat{\Omega}_s(\hat{\mathbf{x}}),\forall s\in[S]\\
		=&~\mathbf{c}^{\intercal}\hat{\mathbf{x}}+v(\hat{\mathbf{x}})\\
		\geq&~\mathrm{UB}.
	\end{align}
	The first equality readily follows from the definition of LB in \eqref{def-lb}. The first inequality is derived since $\widehat{\Omega}_s(\hat{\mathbf{x}})\subseteq\Upsilon_s$ for each $s\in[S]$. The second equality follows from  \eqref{pf-ccg-1}--\eqref{pf-ccg-4}. The second inequality holds due to the definition of UB in \eqref{def-ub}. On the other hand, $\mathrm{LB}\leq \mathrm{UB}$ always holds. Thus, we have $\mathrm{UB}=\mathrm{LB}$.	
\end{proof}

Take $\varepsilon$ as 0. According to the proofs of Corollary \ref{thm-converge-cg} and Claim \ref{cla-ccg}, the number of iterations of C\&CG-DRO(CG) before termination is finite, which is bounded by the number of the extreme points of $\widetilde{\Xi}_s$'s, i.e., $3^{m_{\xi}}\cdot S$.
\end{proof}

\section{Parameters in Numerical Studies}\label{case-set}
The economic and technical parameters of the equipment of MEMG are summarized in Table \ref{para-device}. Table \ref{para-other} contains the other parameters for MEMG's dispatch. The forecasted values and the sample spaces of the random factors are provided in Fig. \ref{fig-space}.

\begin{table}[htbp]
  \centering
  \caption{Equipment Parameters of MEMG}
  \resizebox{1\linewidth}{!}{%
    \begin{tabular}{ccccc}
    \toprule
    Equipment & Parameter &       & Equipment & Parameter \\
\cmidrule{1-2}\cmidrule{4-5}    WT/PV & \makecell[c]{$\overline{P}_{\rm wt}=3000~{\rm kW}$\\
$\overline{P}_{\rm pv}=4000~{\rm kW}$} &       & HT    & \makecell[c]{$\overline{H}_{\rm ht}=300~{\rm kg}$ \\
$\nu_{\rm ht}=2\%$} \\
\cmidrule{1-2}\cmidrule{4-5}    BSS   & \makecell[c]{$\overline{P}_{\rm bss}=2000~{\rm kW}$\\
$\overline{E}_{\rm bss}=4000~{\rm kWh}$\\
$\underline{E}_{\rm bss}=3400~{\rm kWh}$\\
$\eta_{\rm bss,c}=\eta_{\rm bss,d}=90\%$\\
$c_{\rm bss}^{\rm deg}=0.001~{\rm \$/kWh}$} &       & HWT   & \makecell[c]{$\overline{N}_{\rm hwt}=700~{\rm kWh}$\\
$\nu_{\rm hwt}=2\%$} \\
\cmidrule{1-2}\cmidrule{4-5}    ELZ   & \makecell[c]{$c_{\rm elz,c}^{\rm su}=c_{\rm elz,c}^{\rm sd}=10~\$$\\
$c_{\rm elz,w}^{\rm su}=c_{\rm elz,w}^{\rm sd}=0.2~\$$\\
$c_{\rm elz}^{\rm om}=0.001~{\rm \$/kWh}$\\
$\overline{P}_{\rm elz}=1500~{\rm kW}$\\
$\underline{P}_{\rm elz}=225~{\rm kW}$\\
$P_{\rm elz,s}=22.5~{\rm kW}$\\
$\tau_{\rm cold}=1~{\rm h}$\\
$\eta_{\rm elz}=75.66\%$\\
$eta_{\rm rlz,r}=92.03\%$\\
$LHV_{\rm H_2}=33.33~{\rm kW/kg}$} &       & FC    & \makecell[c]{$c_{\rm fc}^{\rm su}=c_{\rm fc}^{\rm sd}=10~\$$\\
$c_{\rm fc}^{\rm om}=0.001~{\rm \$/kWh}$\\
$\overline{P}_{\rm fc}=1000~{\rm kW}$\\
$\underline{P}_{\rm fc}=50~{\rm kW}$\\
$\eta_{\rm fc}=27.27\%$\\
$\eta_{\rm fc,r}=82.14\%$} \\
    \bottomrule
    \end{tabular}%
}
  \label{para-device}%
\end{table}%

\begin{table}[htbp]
  \centering
  \caption{Other Parameters for MEMG's Dispatch}
  \setlength{\tabcolsep}{9.5pt}{%
    \begin{tabular}{ccrcc}
    \toprule
    Parameter & Value &       & Parameter & Value \\
\cmidrule{1-2}\cmidrule{4-5}    $\overline{H}_{\rm buy}$     & $200~{\rm kg}$ &       & $\overline{U}_{\rm g,buy}$     & $2$ \\
    $\overline{P}_{\rm sub}$  & $1500~{\rm kW}$ &       & $\iota_{\rm e}/\iota_{\rm h}$    & $0.2/0.23~{\rm \$/kWh}$ \\
    $P_{\rm d}$    & $3000~{\rm kW}$ &       & $M_{\rm d}$    & $700~{\rm kW}$ \\
    \bottomrule
    \end{tabular}%
}
  \label{para-other}%
\end{table}%

\begin{figure}[h!]
	\centering
	\subfloat[WT outputs.]{
		\label{}
		\includegraphics[width=0.235\textwidth]{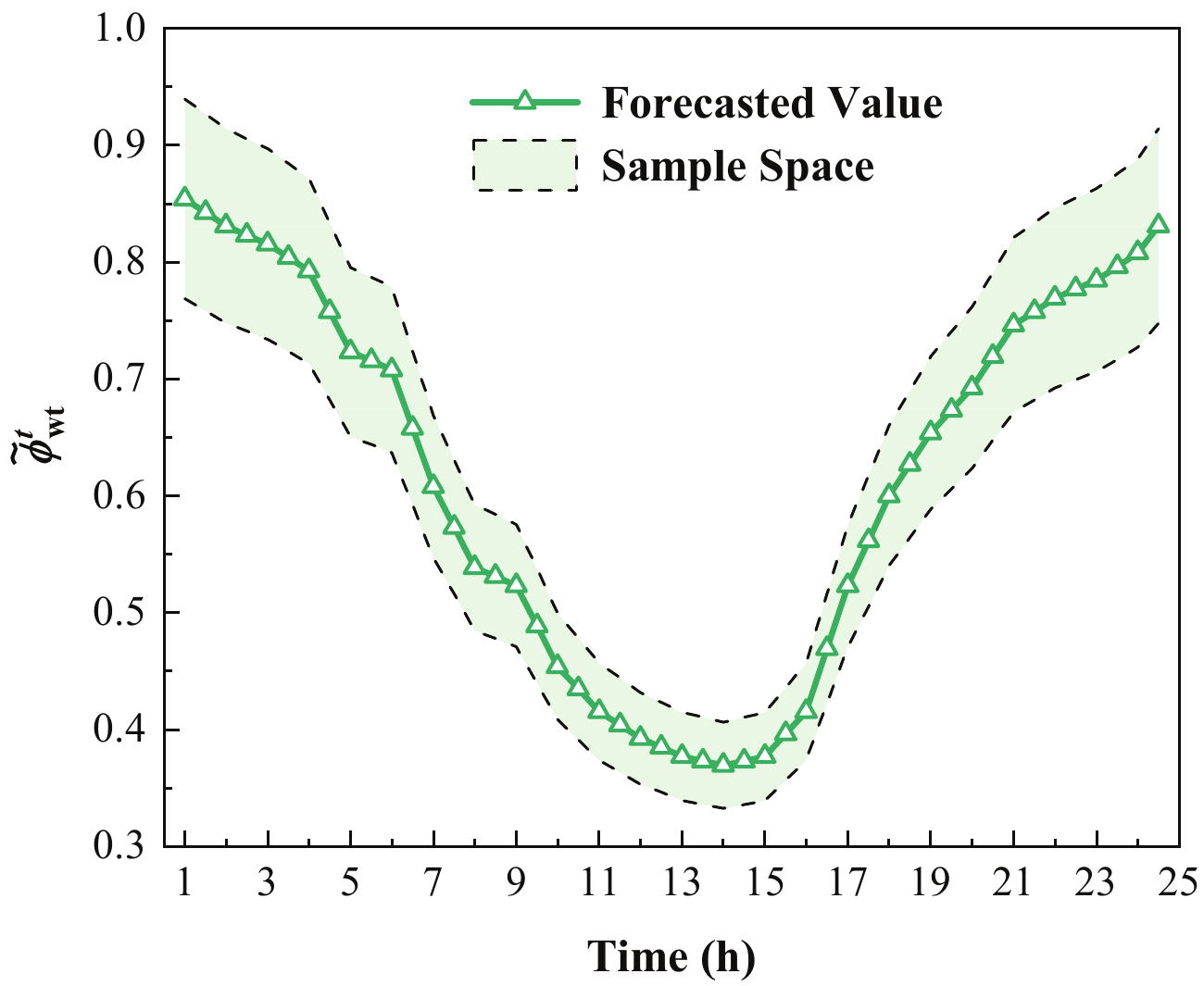}
	}
	\subfloat[PV outputs.]{
		\label{}
		\includegraphics[width=0.235\textwidth]{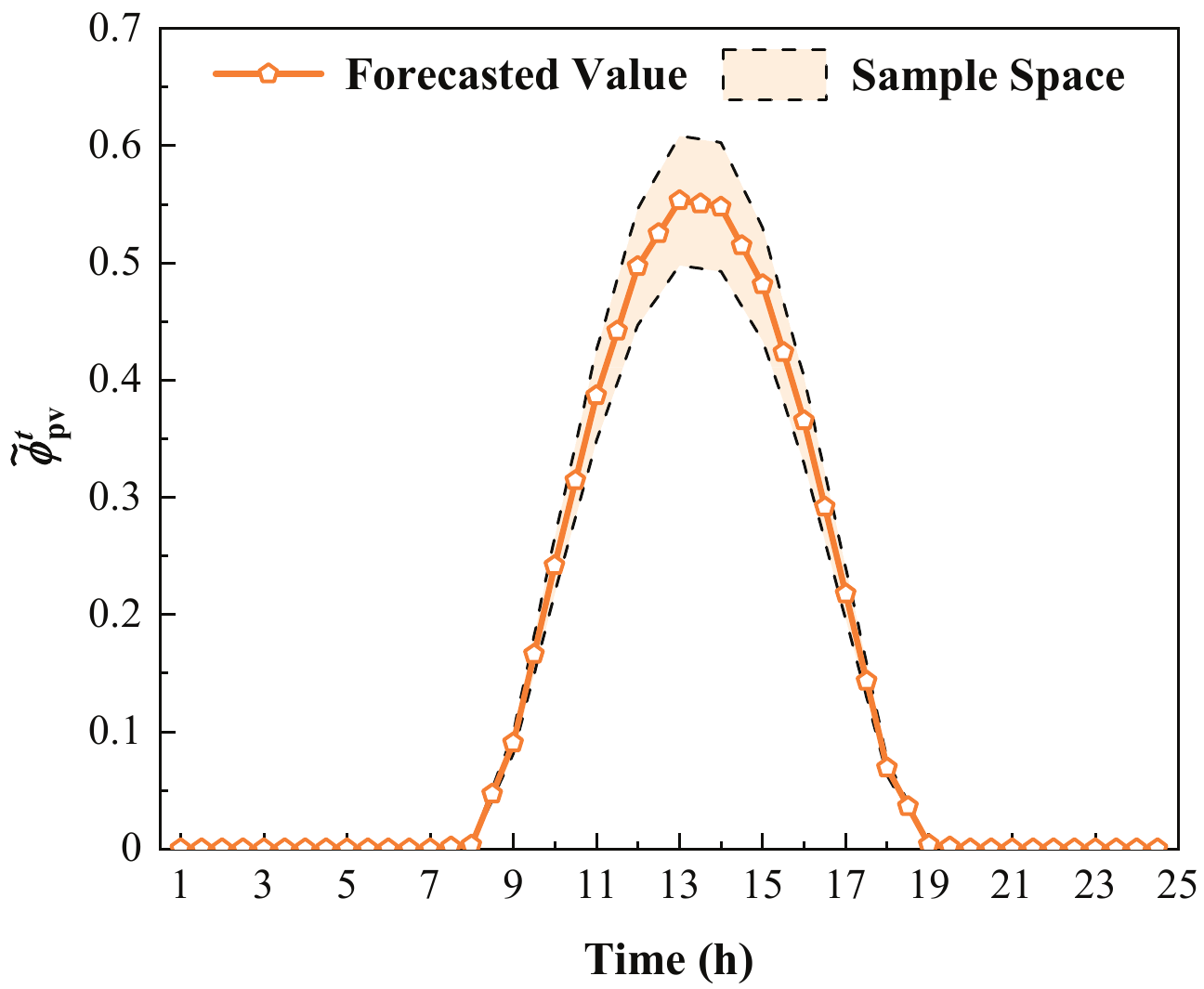}
	}\\
	\subfloat[ELectricity demand.]{
		\label{}
		\includegraphics[width=0.235\textwidth]{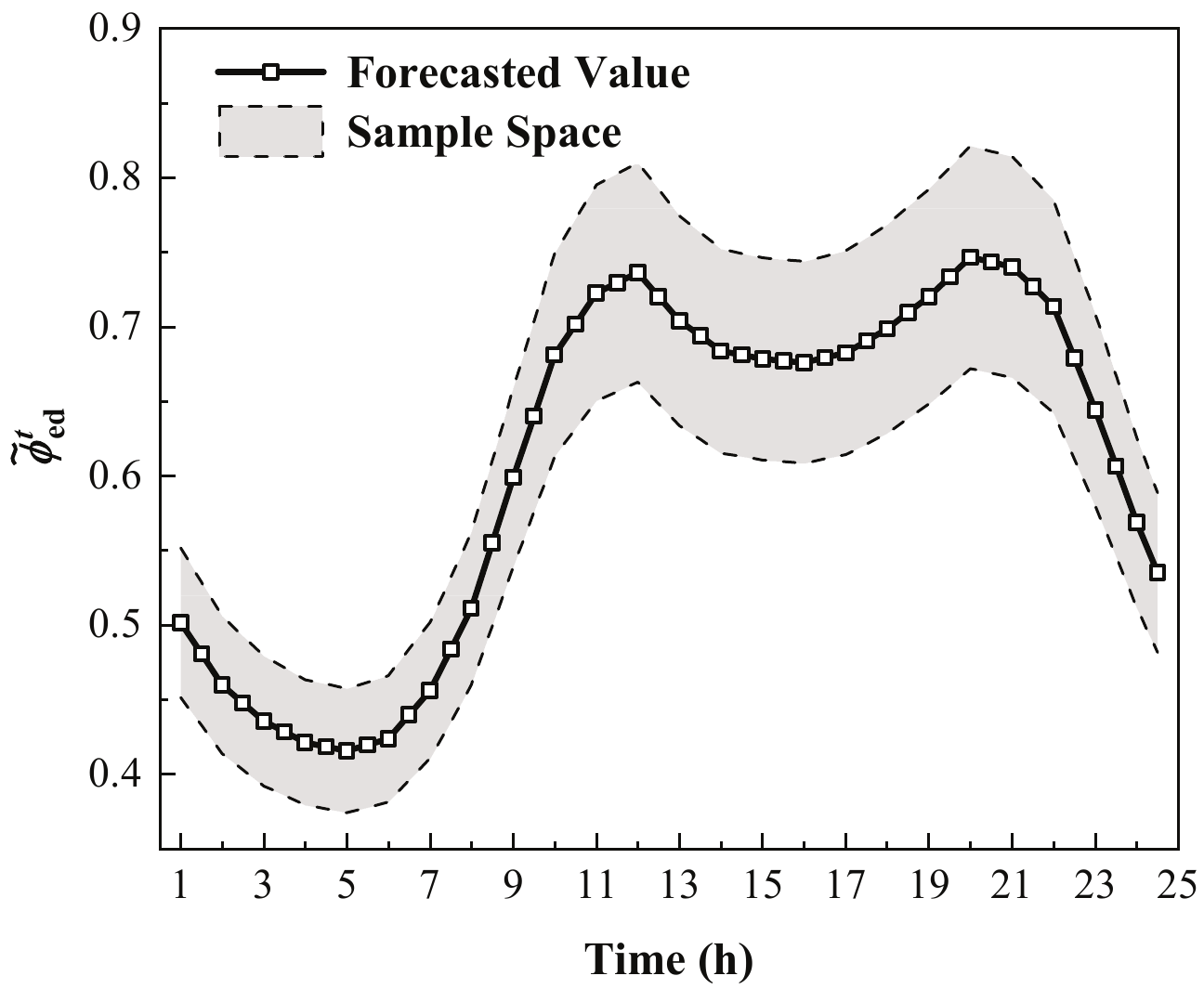}
	}
	\subfloat[Heat demand.]{
		\label{}
		\includegraphics[width=0.235\textwidth]{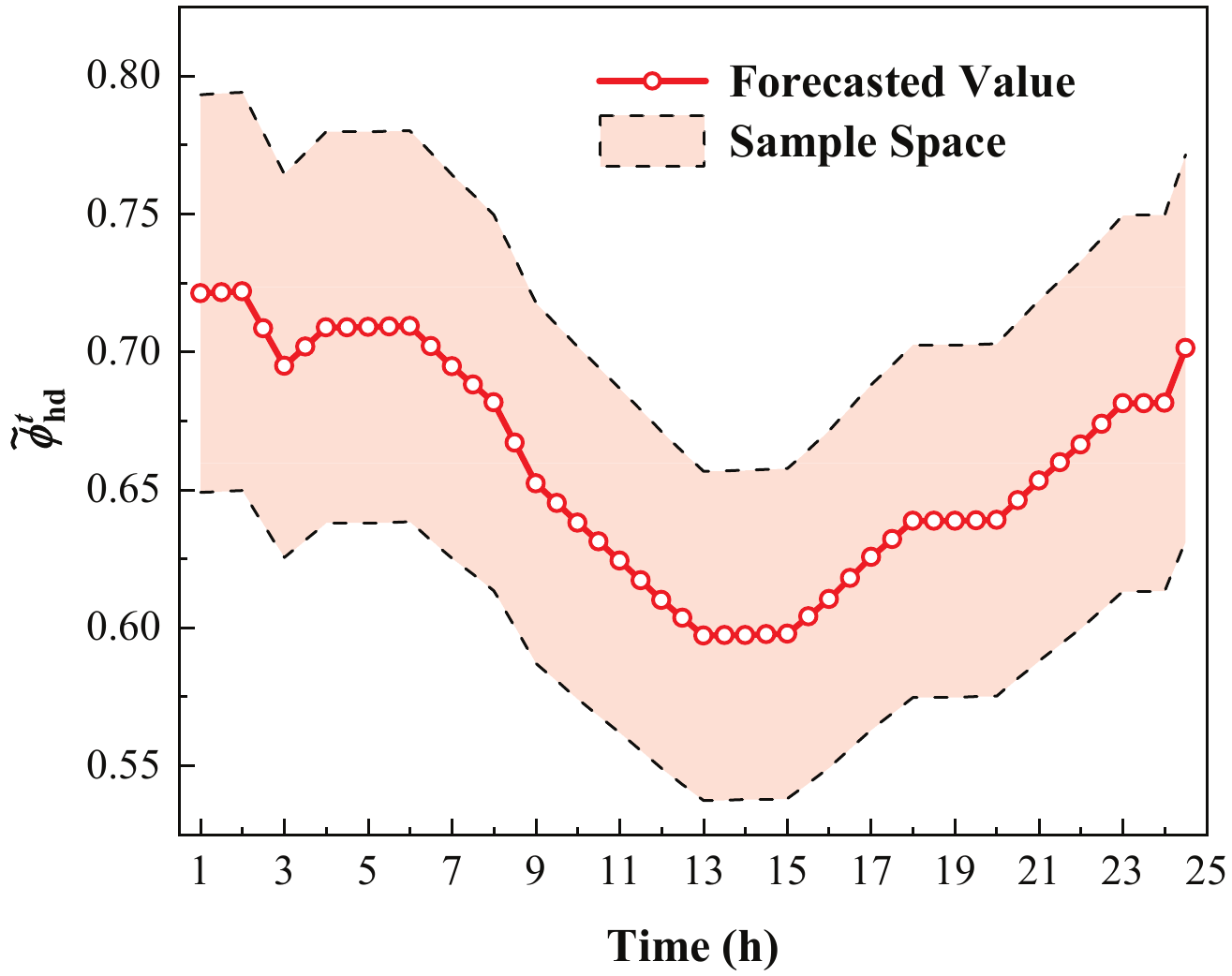}
	}
	\caption{Forecasted values and sample spaces of random factors.} 
	\label{fig-space}
\end{figure}

\section{Out-of-Sample Evaluation Indices}\label{index-oos}
The out-of-sample evaluation is conducted on a test set $\Xi^{\rm o}\triangleq\{\bm{\xi}_1^{\rm o},\cdots,\bm{\xi}_{S^{\rm o}}^{\rm o}\}$ that includes $S^{\rm o}$ scenarios. We define six indices, i.e., the out-of-sample cost (OOSC), probabilities of electricity and heat load shedding (PELS and PHLS), expected electrical and heat energy not supplied (EEENS and EHENS), and expected net CO\textsubscript{2} emissions (ENCE),  as follows.
\begin{align}
	&{\rm OOSC}=\mathbf{c}^{\intercal}\hat{\mathbf{x}}+\frac{1}{S^{\rm o}}\sum_{s=1}^{S^{\rm o}}Q(\hat{\bm{x}},\bm{\xi}_s^{\rm o})\\
	&{\rm PELS}=\frac{1}{S^{\rm o}}\sum_{s=1}^{S^{\rm o}} \mathbbm{1}_{>0}\left(\sum_{t=1}^T\hat{\tilde{p}}_{\rm loss}^{s,t}\right)\times100\%\\
	&{\rm PHLS}=\frac{1}{S^{\rm o}}\sum_{s=1}^{S^{\rm o}} \mathbbm{1}_{>0}\left(\sum_{t=1}^T\hat{\tilde{m}}_{\rm loss}^{s,t}\right)\times100\%\\
	&{\rm EEENS}=\frac{1}{S^{\rm o}}\sum_{s=1}^{S^{\rm o}} \left(\sum_{t=1}^T\hat{\tilde{p}}_{\rm loss}^{s,t}\right)\Delta_t  \\
	&{\rm EHENS}=\frac{1}{S^{\rm o}}\sum_{s=1}^{S^{\rm o}} \left(\sum_{t=1}^T\hat{\tilde{m}}_{\rm loss}^{s,t}\right)\Delta_t\\
	&{\rm ENCE}=\frac{1}{S^{\rm o}}\sum_{s=1}^{S^{\rm o}} \varrho_{\rm co_2}\left(\hat{\tilde{p}}_{\rm buy}^{s,t}-\hat{\tilde{p}}_{\rm sell}^{s,t}\right)\Delta_t
\end{align}
$\hat{\mathbf{x}}$ is the optimal pre-dispatch decisions of \textbf{DRD-MEMG}. $\hat{\tilde{p}}_{\rm loss}^{s,t}$, $\hat{\tilde{m}}_{\rm loss}^{s,t}$, $\hat{\tilde{p}}_{\rm buy}^{s,t}$, and $\hat{\tilde{p}}_{\rm sell}^{s,t}$ are the corresponding optimal solutions to the recourse problem \eqref{2stage-1}--\eqref{2stage-3} with $\hat{\mathbf{x}}$ and scenario $\bm{\xi}_s^{\rm o}$. $\mathbbm{1}_{>0}(\cdot)$ is the characteristic function and
\begin{align}
	\mathbbm{1}_{>0}\left(\sum_{t=1}^T\tilde{p}_{\rm loss}^{s,t}\right)&\triangleq\left\{\begin{matrix}
		1,	\quad{\rm if}\quad  \sum_{t=1}^T\tilde{p}_{\rm loss}^{s,t}>0 \\ 
		0,	\quad{\rm if}\quad  \sum_{t=1}^T\tilde{p}_{\rm loss}^{s,t}=0
	\end{matrix}\right. \\
	\mathbbm{1}_{>0}\left(\sum_{t=1}^T\tilde{m}_{\rm loss}^{s,t}\right)&\triangleq\left\{\begin{matrix}
		1,	\quad{\rm if}\quad  \sum_{t=1}^T\tilde{m}_{\rm loss}^{s,t}>0 \\ 
		0,	\quad{\rm if}\quad  \sum_{t=1}^T\tilde{m}_{\rm loss}^{s,t}=0
	\end{matrix}\right. .
\end{align}
$\varrho_{\rm co_2}=0.5856~{\rm kg}/{\rm kWh}$ is the CO\textsubscript{2} emission factor.

\end{document}